\documentclass[10pt]{article}

\usepackage[english]{babel}

\usepackage[letterpaper,top=2cm,bottom=2cm,left=3cm,right=3cm,marginparwidth=1.75cm]{geometry}

\usepackage{amsmath,amssymb,amsfonts,amsthm}
\usepackage{graphicx}
\usepackage[colorlinks=true, allcolors=blue]{hyperref}
\usepackage{bussproofs} 
\usepackage{multicol}
\usepackage{color}

\newtheorem{definition}{Definition}[section]
\newtheorem{theorem}[definition]{Theorem}
\newtheorem{lemma}[definition]{Lemma}
\newtheorem{proposition}[definition]{Proposition}
\newtheorem{remark}[definition]{Remark}
 
\newtheorem{example}[definition]{Example}
\newtheorem{corollary}[definition]{Corollary}

\def\A{\mathcal{A}}

\def\P{\mathbf{P}}
\def\F{\mathbf{F}}
\def\H{\mathbf{H}}
\def\G{\mathbf{G}}
\def\K{\mathcal{K}}
\def\La{\mathcal{L}}

\def\tdlat{{\bf tDL}}
\def\dl{{\bf DL}}
\def\tps{{\bf tPS}}
\def\ps{{\bf PS}}

\title{Tense distributive lattices: algebra, logic and topology.}
\author{A. V. Figallo\footnote{E-mail: \texttt{avfigallo@gmail.com}},  J. Sarmiento\footnote{E-mail: \texttt{jsarmiento33@gmail.com }}\, and  M. Figallo\footnote{E-mail: \texttt{figallomartin@gmail.com}}  \\ [2mm] %
{\small $^\ast$$^\dag$ Instituto de Ciencias B\'asicas, Universidad Nacional de San Juan (UNSJ), San Juan, Argentina}\\
{\small $^\dag$$^\ddag$ Departamento de Matemática, Universidad Nacional del Sur, Bahía Blanca, Argentina}\\
{\small $^\ddag$ Departamento de Matemática and Instituto de Matemática (INMABB),}\\ 
{\small Universidad Nacional del Sur, Bahía Blanca, Argentinal}}
\date{}

\begin{document}
\maketitle

\begin{abstract}
Tense logic was introduced by Arthur Prior in the late 1950’s as a result of his interest in the relationship between tense and modality. Prior's idea was to add four primitive modal-like unary connectives to the base language today widely known as {\em Prior’s tense operators}. Since then, Prior's operators have been considered in many contexts by different authors, in particular, in the context of algebraic logic.

Here, we consider the category \tdlat\ of bounded distributive lattices equipped with Prior’s tense operators. We establish categorical dualities for \tdlat\ in terms of certain categories of Kripke frames and Priestley spaces, respectively. As an application, we characterize the congruence lattice of any tense distributive lattice as well as the subdirectly irreducible members of this category. Finally, we define the logic that preserves degrees of truth with respect to \tdlat-algebras and precise the relation between particular sub-classes of \tdlat\ and  know tense logics found in the literature.
\end{abstract}

\vspace*{4mm}

\noindent {\bf MSC (2010):} {Primary 06D35, Secondary 03B60.} 

\vspace*{4mm}

\noindent {\bf \em Keywords:} {tense algebraic structures, topological dualities, logics that preserve degrees of truth.}

\section{Introduction}
The term {\em Temporal Logic} is used to designate a wide range of logical systems, and their formal representation, which allow reasoning about time and temporal information. In particular, it is used to refer to the approach introduced by Arthur Prior under the name {\em Tense Logic} and subsequently developed further by many logicians and computer scientists. The contribution of Prior was to consider temporal logics as a special kind of modal logic where truth values of assertions vary with time. 
Temporal logics have many applications such as providing a suitable formalism for clarifying philosophical issues about time as well as within which to define semantics of temporal expressions in natural language. Besides, it provides the suitable language for encoding temporal knowledge in artificial intelligence, and as a tool for specification, formal analysis, and verification of the executions of computer programs and systems.
Prior's motivation for considering his well-known tense operators $\P$, $\F$, $\H$ and $\G$ was strongly motivated in the use of tense in natural language (see \cite{Prior57, Prior67, Prior68}). The intended meaning of these four temporal operators is the following: $\P$
:``{\em It has at some time been the case that $\dots$}''; $\F$ :``{\em It will at some time be the case that $\dots$}''; $\H$ :``{\em It has always been the case that $\dots$}''; and $\G$ :``{\em It will always be the case that $\dots$}''.\\
The {\em minimal tense logic} $K_t$ is the logic whose axioms are all classical tautologies along with the axioms

$$\G(\varphi\rightarrow \psi)\rightarrow (\G\varphi\rightarrow \G\psi),$$

$$\H(\varphi\rightarrow \psi)\rightarrow (\H\varphi\rightarrow \H\psi),$$

$$\varphi\rightarrow\G\P\varphi,$$

$$\varphi\rightarrow\H\F\varphi.$$

\noindent and the inference rules are the classical {\em modus ponens} and 

$$\displaystyle \frac{\varphi}{\G\varphi},$$

$$\displaystyle \frac{\varphi}{\H\varphi}.$$

The first two axiom schemata are the temporal correspondents of the so-called $K$-axiom of modal logic, and hence the terminology $K_t$. The third and the fourth axiom schemata capture the interaction of the past and future operators. The last two rules are the corresponding {\em necessitation rules} for the temporal operators.
$K_t$ encompass one pair of temporal operators for the past, $\P$ and $\H$, and one pair of temporal operators for the future, $\F$ and $\G$. The operators $\P$ and $\F$ are often referred to as the ‘weak’ temporal operators, while $\H$ and $\G$  are known as the ‘strong’ ones. The respective past and future operators are duals of each other, i.e., they are interdefinable by means of the following equivalences:

$$\P\varphi\equiv \neg\H\neg\varphi, \,  \H\varphi\equiv\neg\P\neg\varphi \, \mbox{ and } \, \F\varphi\equiv \neg\G\neg\varphi, \, \G\varphi\equiv\neg\F\neg\varphi$$

Later, in \cite{Ewald86}, intuitionistic tense logic $IK_t$ was introduced by Ewald by extending the language of intuitionistic propositional logic with the four tense operators. A Hilbert-style axiomatization of $IK_t$ can be found in \cite{Ewald86}. It is well-known that the Ewald's axiomatization of is not minimal in the sense that several axioms can be deduced from others. Besides, in contrast to what happens in $K_t$, in the intuitionistic case $\F$ and $\P$ cannot be defined in terms of $\G$ and $\H$ (see \cite{Dzik}). In order to provide a semantics for $IK_t$, Ewald introduced the notion of {\em intuitionistic tense structure}. Fromally, a intuitionistic tense structure is a quintuple $\langle \Gamma, \leq, \{T_\gamma\}_{\gamma\in \Gamma}, \{u_\gamma\}_{\gamma\in \Gamma}, \{R_{t}^{\gamma}\}_{\gamma\in \Gamma, t\in T_{\gamma}} \rangle$ where $(\Gamma, \leq)$ is a partially-ordered set (the "states-of-knowledge"), $T_{\gamma}$, is a set (the set of times known at state-of-knowledge $\gamma$), and $u_\gamma$ is a binary relation on $T_{\gamma}$ (the temporal ordering of $T_\gamma$ as it is understood at state-of-knowledge $\gamma$). We require that $\gamma\leq\varphi$  imply $T_{\gamma} \subseteq T_{\varphi}$ and and $u_\gamma \subseteq u_{\varphi}$, that is, if we advance in knowledge we retain what we know about times and their temporal ordering. On the other hand, if we require times and the temporal ordering to be unchanging, that is,
for all $\gamma$, $\varphi$ we have $T_\gamma = T_\varphi$, and $u_\gamma = u_\varphi$, we obtain the {\em intuitionistic tense logic of unchanging times} denoted by $IK^{*}_t$; and where each $R_{t}^{\gamma}$ is a relation on formulas such that (see \cite{Ewald86}): 
\begin{enumerate}
\item $R_{t}^{\gamma}$(p) and $\gamma\leq \varphi$ imply $R_{t}^{\varphi}(p)$ for atomic $p$; 
\item $R_{t}^{\gamma}(X\wedge Y)$ iff $R_{t}^{\gamma}(X)$ and $R_{t}^{\gamma}(Y)$; 
\item $R_{t}^{\gamma}(X \vee Y)$ iff $R_{t}^{\gamma}(X)$ or $R_{t}^{\gamma}(Y)$; 
\item $R_{t}^{\gamma} (\neg X)$ iff ($\forall \varphi: \varphi\geq \gamma$) not $R_{t}^{\varphi}(X)$; 
\item $R_{t}^{\gamma}(X\rightarrow Y)$ iff ($\forall \varphi: \varphi\geq \gamma$)($R_{t}^{\varphi}(X)$ implies $R_{t}^{\varphi}(Y))$; 
\item $R_{t}^{\gamma}(\P X)$ iff ($\exists t'$)($t'u_{\gamma}t$ and $R_{t'}^{\gamma}(X)$); 
\item $R_{t}^{\gamma}(\F X)$ iff ($\exists t'$)($tu_{\gamma}t'$ and $R_{t'}^{\gamma}(X)$); 
\item $R_{t}^{\gamma}(\H X)$ iff ($\forall \varphi \geq \gamma$)($\forall t'\in T_{\varphi}$)($t'u_{\varphi}t$ implies $R_{t'}^{\varphi}(X)$); 
\item $R_{t}^{\gamma}(\G X)$ iff ($\forall \varphi \geq \gamma$)($\forall t'\in T_{\varphi}$)($tu_{\varphi}t'$ implies $R_{t'}^{\varphi}(X)$).
\end{enumerate}

\

From the algebraic point of view, recall that {\em tense algebras} are Boolean algebras equipped with two unary join-preserving and normal operators $\F$ and $\P$ which are conjugates, that is, they satisfy  $$x\wedge \P y = 0 \, \mbox{ iff } \, \F x \wedge y=0.$$ In every tense algebra one can consider the operators $\H$ and $\G$ defined as $\H\varphi\equiv\neg\P\neg\varphi$ and $\G\varphi\equiv\neg\F\neg\varphi$.
It is well-known that tense algebras are an algebraic counterpart of $\K_t$.\\
In \cite{F2}, it was considered the class of Heyting algebras equipped with four unary operators $\F$, $\P$, $\H$ and $\G$ where the first two operators are join-preserving, normal and conjugates, the last two verifies the dual properties and such that
$$\F(x\rightarrow y)\leq \G x \rightarrow \F y \, \mbox{ \, and \, } \, \P(x\rightarrow y)\leq \H x \rightarrow \P y$$
Besides, it was proved that this class of algebras, known as $IK_t$-algebras, constitutes an algebraic counterpart of $IK_t$.
In the last few years, tense operators have been considered by different
authors in diverse classes of algebras. Some contributions in this area can be found in works by Diaconescu and Georgescu \cite{Diaconescu07}, Botur et al. \cite{Botur11}, Chiri\c{t}\u{a} \cite{Chirita11},
Chajda \cite{Chajda11}, Figallo, Pascual and Pelaitay \cite{F1,F2,F6}, Bakhshi \cite{BM}, Botur and Paseka \cite{BP}, Menni and Smith \cite{Menni14}, Dzik et al. \cite{Dzik}, Figallo, Pelaitay and Sarmiento \cite{F21,F22,F23}.

\

In this work, we focus on topological and logical aspects of tense operators defined over bounded distributive lattices. These structures have been considered already in works such as \cite{ChajdaPaseka} and \cite{Menni14}. In Section \ref{s2}, we introduce the category \tdlat\ of tense distributive lattices, we study the objects (\tdlat-algebras) of this category; in particular, we prove important properties of the prime filters of a given \tdlat-algebra. In Section \ref{s3}, we develop a topological duality for tense distributive lattices, more precisely, we introduce the category \tps\ of tense Priestley spaces, whose objects are Priestley spaces equipped with a particular binary relation, and prove that the categories \tdlat\ and \tps\ are naturally equivalent. As an application of this, in Section \ref{s4}, we give a characterization of the simple and subdirectly irreducible objects of \tdlat. Also, we particularize many of these results to some well-known subcategories of \tdlat. Later, we develop a discrete duality for \tdlat-algebras in Section \ref{s5}. Finally, in Section \ref{s6}, we define the {\em logic that preserves degrees of truth} w.r.t. \tdlat-algebras and other tense structures, provide syntactic presentations for them and show that some of them coincide with well-known tense logics. 

\section{Tense distributive lattices}\label{s2}
For convenience, all lattices in this section are assumed to be bounded and distributive and we denote by \dl\ the category whose objects are bounded distributive lattices and whose morphisms are the usual homomorphisms between lattices. Let us recall the notion of tense distributive lattices:

\begin{definition} \label{D21} A {\em tense distributive lattice} is a structure $\mathcal{A}=\langle \mathcal{A}_0, \G, \H, \F, \P\rangle$ where $\mathcal{A}_0=\langle A, \wedge, \vee, 0, 1\rangle$ is a bounded distributive lattice and $\G,\H,\F$ and $\P$ are Prior's tense operators defined on $\mathcal{A}_0$, that is, they are unary operators satisying:
\begin{multicols}{2}
\begin{itemize}
  \item[\rm{(t1)}] $\G 1=1$ and $\H 1=1$,
  \item[\rm{(t2)}] $\G(x\wedge y)=\G x\wedge \G y$ and $\H(x\wedge y)=\H x\wedge \H y$,
  \item[\rm{(t3)}] $x\leq \G\P x$ and $x\leq \H\F x$,
  \item[\rm{(t4)}] $\G(x\vee y)\leq \G x\vee \F y$ and $\H(x\vee y)\leq \H x\vee \P y$,
  \item[\rm{(t5)}] $\F 0=0$ and $\P 0=0$,
  \item[\rm{(t6)}] $\F(x\vee y)=\F x\vee \F y$ and $\P(x\vee y)=\P x\vee \P y$,
  \item[\rm{(t7)}] $\P\G x\leq x$ and $\F\H x\leq x$,
  \item[\rm{(t8)}] $\G x\wedge \F y\leq \F(x\wedge y)$ and $\H x\wedge \P y\leq \P(x\wedge y)$.
\end{itemize}
\end{multicols}
\end{definition}

Notice that, from the point of view of Universal Algebra, the class of tense distributive lattices constitute a variety.
We denote by \tdlat\ the category whose objects are tense distributive lattices and whose morphisms are tense distributive homomorphisms, that is, lattice homomorphisms that respect the four tense operators. We call the objects of \tdlat\ ``\tdlat-algebras''. 
\begin{remark} A distributive lattice with adjunction (dLata) is a triple $\langle \mathcal{A}, L, R\rangle$ (also noted $\langle \mathcal{A}, L\dashv R\rangle$) where $\mathcal{A}$ is a distributive lattice and $L, R : A \rightarrow A$ are monotone functions such that $L$ is left adjoint to $R$. In \cite{Menni14}, it was introduced {\em tense dLatas} as structures $\langle\mathcal{A}, \Diamond\dashv H, P\dashv\square\rangle$ where both  $\langle\mathcal{A}, \Diamond\dashv H\rangle$ and  $\langle\mathcal{A}, P\dashv\square\rangle$ are dLatas and such that
$$\Diamond p \wedge \square q \leq \Diamond (p\wedge q) \hspace{1cm} Pp \wedge Hq \leq P(p\wedge q)$$
and 
$$\square(p \vee q) \leq \square p \vee \Diamond q \hspace{1cm} H(p \vee q) \leq Hp\vee Pq$$
hold. It is clear that the notions of tense dLata and tense distributive lattices coincide (see Lemma \ref{L212}.
\end{remark}

\

In every \tdlat-algebra it is possible to define two particular unary operations which are very useful when characterizing important notions:

\begin{definition}\label{Dd} Let $\mathcal{A}$ a \tdlat-algebra. Define the functions $d,\hat{d}:A\longrightarrow  A$ by
\begin{equation}
d x=\G x\wedge x\wedge \H x
\end{equation}
and
\begin{equation}
\hat{d} x= \F x\vee x\vee \P x
\end{equation}
for every $x\in A$. Besides, for $n\in \omega$, define $d^n x$ and $\hat{d}^n x$ inductively by

\begin{equation}
d^0 x=x \mbox{ \, and \, } d^{n+1} x=dd^n x
\end{equation}
and
\begin{equation}
\hat{d}^0 x=x \mbox{ \, and \, } \hat{d}^{n+1} x=\hat{d}\hat{d}^nx.
\end{equation}
\end{definition}
\noindent Besides, given $X\subseteq A$ and $n\in\omega$, we denote $d^n X$ and $\hat{d}^n X$ the sets 
$$d^n X:=\{d^n x: x\in X\} \mbox{ \, and \, }\hat{d}^n X:=\{\hat{d}^n x: x\in X\}.$$

\noindent Next, we show some examples of tense distributive lattices
\begin{example} There are two extreme examples of tense operators that can be considered on a given distributive lattice $\mathcal{A}$: (1) define $\G,\H,\F$ and $\P$  as the identity function $id_{A}$; (2) define $\G$ and $\H$ as the constant function $1_A$ {\rm (}$1_A x = 1$, for every $x\in A${\rm )}; and $\F$ and $\P$ as the constant function $0_A$.
\end{example}


\begin{example}\label{ej2} Let $\mathcal{A}_0$ the lattice whose diagram is the following:
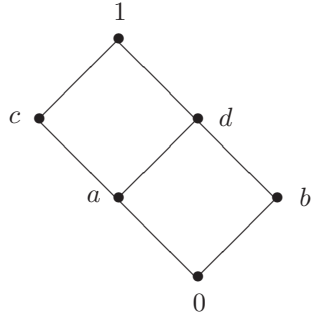
\vspace{2.5cm}
\begin{figure}[htbp]
\begin{center}
\hspace{0.25cm}
\begin{picture}(-40,40)(0,0)
\put(00,00){\makebox(1,1){$\bullet$}}
\put(-30,30){\makebox(1,1){$\bullet$}}
\put(30,30){\makebox(1,1){$\bullet$}}
\put(00,60){\makebox(1,1){$\bullet$}}
\put(-60,60){\makebox(1,1){$\bullet$}}
\put(-30,90){\makebox(1,1){$\bullet$}}
\put(00,00){\line(1,1){30}}
\put(00,00){\line(-1,1){30}}
\put(-30,30){\line(1,1){30}}
\put(-30,30){\line(-1,1){30}}
\put(30,30){\line(-1,1){30}}
\put(-60,60){\line(1,1){30}}
\put(00,60){\line(-1,1){30}}
\put(00,-10){\makebox(2,2){$ 0$}}
\put(-40,30){\makebox(2,2){$ a$}}
\put(40,30){\makebox(2,2){$ b$}}
\put(-70,60){\makebox(2,2){$ c$}}
\put(10,60){\makebox(2,2){$ d$}}
\put(-30,100){\makebox(2,2){$ 1$}}
\end{picture}
\end{center}
\end{figure}

\

Define the operators $\G,\H, \F$ and $\P$ as follows 

\begin{center}
\begin{tabular}{|c|c|c|c|c|}\hline
   $x$   &  $\G x$ & $\H x$ &  $\F x$ & $\P x$ \\ \hline
   $0$   &   $b$   &   $0$  &   $0$   &   $0$  \\ \hline
   $a$   &   $b$   &   $a$  &   $a$   &   $d$  \\ \hline
   $b$   &   $b$   &   $0$  &   $c$   &   $0$  \\ \hline
   $c$   &   $b$   &   $1$  &   $c$   &   $1$  \\ \hline
   $d$   &   $d$   &   $a$  &   $c$   &   $d$  \\ \hline
   $1$   &   $1$   &   $1$  &   $c$   &   $1$  \\ \hline
\end{tabular}
\end{center}

\noindent Then, it can be verified that $\langle\mathcal{A}_0,\G,\H,\F,\P\rangle$ is a  \tdlat-algebra and, as we shall see, it is a simple algebra.
\end{example}


\begin{proposition}\label{P26} Let $\mathcal{A}$ be a \tdlat-algebra. Then
\begin{multicols}{2}
\begin{itemize}
  \item[\rm{(t9)}] $\G, \H, \F$ and $\P$ are monotone.
  \item[\rm{(t10)}] $\G x\vee \G y\leq \G(x\vee y)$  and $\H x\vee \H y\leq \H(x\vee y)$
  \item[\rm{(t11)}] $\F(x\wedge y)\leq \F x\wedge \F y$ and $\P(x\wedge y)\leq \P x\wedge \P y$,  
  \item[\rm{(t12)}] $x\wedge \F y\leq \F(\P x\wedge y)$ and $x\wedge \P y\leq \P(\F x\wedge y)$,
  \item[\rm{(t13)}] $\F x\wedge y=0$ \, iff \, $x\wedge \P y=0$,
  \item[\rm{(t14)}] $\G(x\vee \H y)\leq \G x\vee y$ and $\H(x\vee \G y)\leq \H x\vee y$,
  \item[\rm{(t15)}] $x\vee \G y=1$ \, iff \, $\H x\vee y=1$,   
  \item[\rm{(t16)}] $\P x\leq y$ \, iff \, $x\leq \G y$,
  \item[\rm{(t17)}] $\F x\leq y$ \, iff \, $x\leq \H y$,
  \item[\rm{(t18)}] $\F=\F\H\F$, $\P=\P\G\P$, $\G=\G\P\G$ and $\H=\H\F\H$.
\end{itemize}
\end{multicols}
\end{proposition}
\begin{proof} 

From $x\leq y$ iff $x=x\wedge y$ iff $x\vee y=y$ we have that: (t9) follows from (t2) abd (t6), besides,  (t10) and (t11) are consequence of (t9). Let us prove (t12): $x\wedge \F y\leq_{(t3)}\G\P x\wedge \F y\leq_{(t8)}\F(\P x\wedge y)$, (t13) is consequence of (t12) and (t5).  (t14) follows from $\G(x\vee \H y)\leq_{(t4)}\G x\vee \F\H y\leq_{(t7)}\G x\vee y$, (t15) is consequence of (t14) and (t1). (t16) is proved from:  $\P x\leq y$ implies$_{(t9)}$ $\G\P x\leq \G y$ implies$_{(t3)}$ $ x\leq \G y$, conversely, $x\leq \G y$ implies$_{(t9)}$ $\P x\leq \P\G y$ implies$_{(t7)}$ $\P x\leq y$. The proof of (t17) is similar to the one of (t16). Finally, for (t18) we only show $\F=\F\H\F$ (the others are similar), from (t3) $x\leq \H\F x$ then, by  (t9), $\F x\leq \F\H\F x$, besides from (t7) $\F\H\F x\leq \F x$.
\end{proof}

\begin{proposition}\label{P27} Let $\mathcal{A}$ be a \tdlat-algebra and let $\{a_i\}_{i\in I}\subseteq A$. If $\displaystyle\bigwedge_{i\in I}a_i$ and  $\displaystyle\bigvee_{i\in I}a_i$ exist then:
\begin{multicols}{2}
\begin{itemize}
  \item[(i)] $\displaystyle\bigwedge_{i\in I}\G a_i$ exists and $\displaystyle\bigwedge_{i\in I}\G a_i=\G \displaystyle\bigwedge_{i\in I}a_i$,
  \item[(ii)] $\displaystyle\bigwedge_{i\in I}\H a_i$ exists and $\displaystyle\bigwedge_{i\in I} \H a_i=\H \displaystyle\bigwedge_{i\in I}a_i$, 
  \item[(iii)] $\displaystyle\bigwedge_{i\in I}d a_i$ exists and $\displaystyle\bigwedge_{i\in I}d a_i=d \displaystyle\bigwedge_{i\in I}a_i$,
  \item[(iv)] $\displaystyle\bigvee_{i\in I}\F a_i$ exists and $\displaystyle\bigvee_{i\in I}\F a_i=\F \displaystyle\bigvee_{i\in I}a_i$, 
  \item[(v)] $\displaystyle\bigvee_{i\in I}\P a_i$ exists and $\displaystyle\bigvee_{i\in I}\P a_i= \P\displaystyle\bigvee_{i\in I}a_i$, 
  \item[(vi)] $\displaystyle\bigvee_{i\in I}\hat{d} a_i$ exists and $\displaystyle\bigvee_{i\in I}\hat{d} a_i =\hat{d}\displaystyle\bigvee_{i\in I}a_i$.
\end{itemize}
\end{multicols}
\end{proposition}
\begin{proof} We only show (i), the rest are similar. If $\displaystyle\bigwedge_{i\in I}a_i$ exists so does $\G\displaystyle\bigwedge_{i\in I}a_i$. Besides, by (t9) we have that $\G\displaystyle\bigwedge_{i\in I}a_i$ is lower bound of $\{\G a_i\}_{i\in I}$. Let $b\in A$ such that $b\leq \G a_i$ for every $i\in I$, then by (t16), $\P b\leq  a_i$ for all $i\in I$, and therefore $\P b\leq \displaystyle\bigwedge_{i\in I}a_i$, again from (t16) we have $b\leq \G \displaystyle\bigwedge_{i\in I}a_i$.    
\end{proof}

From Definition \ref{Dd} and axioms (t1), (t2), (t5) and (t6) we can easily prove the following properties using induction on $n\in \omega$.  

\begin{proposition}\label{P28} Let $\mathcal{A}$ be a \tdlat-algebra. Then for all $n\in \omega$ 
\begin{multicols}{2}
\begin{itemize}
  \item[$(d_1)$] $d^n 1 =1$ and $d^n 0 =0$,
  \item[$(d_2)$] $d^{n+1} x \leq d^n x $,
  \item[$(d_3)$] $d^{n}(x\wedge y)=d^{n} x \wedge d^{n} y $,
  \item[$(d_4)$] $x\leq y$ implies $d^{n} x \leq d^n y $,
  \item[$(\hat{d}_1)$] $\hat{d}^n 1 =1$ and $\hat{d}^n 0 =0$,
  \item[$(\hat{d}_2)$] $\hat{d}^{n} x \leq \hat{d}^{n+1} x $,
  \item[$(\hat{d}_3)$] $\hat{d}^{n}(x\vee y)=\hat{d}^{n} x \vee \hat{d}^{n} y $,
  \item[$(\hat{d}_4)$] $x\leq y$ implies $\hat{d}^{n} x \leq \hat{d}^n y $.
\end{itemize}
\end{multicols}
\end{proposition}

\begin{proposition}\label{P29} Let $\mathcal{A}$ be a \tdlat-algebra and $x,y\in A$. Then
\begin{multicols}{2}
\begin{itemize}
  \item[(a)]  $d^n x \leq d x \leq x\leq \hat{d} x \leq \hat{d}^n x $,
  \item[(b)]  $x\leq d\hat{d} x$ y $\hat{d} d x\leq x$,
  \item[(c)]  $\hat{d} x \leq y$ iff $x\leq d y $,
  \item[(d)]  $x=d x $ iff $x=\hat{d} x $,
  \item[(e)]  $x=d x $ iff there exists  $n\in \omega$ such that $x\leq d^n x $,
  \item[(f)]  $x=d x $ iff  there exists $n\in \omega$ such that $\hat{d}^n x \leq x$.
\end{itemize}
\end{multicols}
\end{proposition}

\begin{proof} We only prove (b) and (d).\\
(b): It is clear that $\P x\leq \hat{d} x $ and $\F x\leq \hat{d} x $ then, by (t9) and (t3) we have $x\leq \G\P x\leq \G \hat{d} x$ and $x\leq \H\F x\leq \H \hat{d} x$. Besides, $x\leq \hat{d} x $ and therefore $x\leq \G \hat{d} x\wedge\hat{d} x \wedge \H \hat{d} x=d \hat{d} x$. Similarly, we prove $\hat{d} d x \leq x$.\\[2mm]
$(d):$ $x=d x $ implies $\hat{d} x =\hat{d} d x \leq_{(b)}x\leq_{(a)}\hat{d} x $ and then $\hat{d} x =x$. Conversely, $x=\hat{d} x $ implies $d x =d\hat{d} x\geq_{(b)}x\geq_{(a)}d x $ and then $d x =x$.
\end{proof}

\begin{definition}\label{Dinva} Given the \tdlat-algebra $\mathcal{A}$, we say that $x\in A$ is a $d$-invariant element of $\mathcal{A}$ if $x=d(x)$. We denote by $A^d$ the set of all $d$-invariant elements of $\mathcal{A}$. That is, $A^d:=\{x\in A: x=d(x)\}$.
\end{definition}

\begin{lemma}\label{L211} Let $\mathcal{A}$ be a \tdlat-algebra. Then, the structure $\mathcal{A}^d=\langle A^{d}, \wedge, \vee, 0,1\rangle$ is a sublattice of  $\mathcal{A}_0$.
\end{lemma}
\begin{proof} From Proposition \ref{P28} $(d_1)$, we have $\{0,1\}\subseteq A^d$. Let $x,y\in A^d$, that is, $d x =x$ and $d y =y$ then by $(d_3)$ we have $x\wedge y\in A^d$. On the other hand, from Proposition \ref{P29} (d) we have that $\hat{d} x =x$ and $\hat{d} y =y$. Then $x\vee y=\hat{d} x \vee \hat{d} y =_{(\hat{d}_3)}\hat{d}(x\vee y)$, and then $x\vee y=d(x\vee y)$ (by Proposition \ref{P29} (d)). That is,  $x\vee y\in A^d$.
\end{proof}

\

\noindent The next lemma provides alternative ways to define \tdlat-algebras.

\begin{lemma}\label{L212} Let $\mathcal{A}_0=\langle A,\wedge, \vee, 0, 1\rangle$ be a distributive lattice and $\G, \H, \F, \P$ be unary operators on $A$. The following conditions are equivalent.
\begin{itemize}
  \item[(a)] $\langle\mathcal{A}, \G, \H, \F, \P\rangle$ is a \tdlat-algebra. 
  \item[(b)] It holds 
  \begin{itemize}  
  \item[\rm{(t4)}] $\G(x\vee y)\leq \G x\vee \F y$ and $\H(x\vee y)\leq \H x\vee \P y$, 
  \item[\rm{(t8)}] $\G x\wedge \F y\leq \F(x\wedge y)$ and $\H x\wedge \P y\leq \P(x\wedge y)$
  \item[\rm{(t16)}] $\P x\leq y$ \, iff \, $x\leq \G y$,
  \item[\rm{(t17)}] $\F x\leq y$ \, iff \, $x\leq \H y$,
  \end{itemize} 
  \item[(c)] It holds
  \begin{itemize} 
  \item[\rm{(t4)}] $\G(x\vee y)\leq \G x\vee \F y$ and $\H(x\vee y)\leq \H x\vee \P y$, 
  \item[\rm{(t8)}] $\G x\wedge \F y\leq \F(x\wedge y)$ and $\H x\wedge \P y\leq \P(x\wedge y)$
  \item[\rm{(t9)}] $\G, \H, \F$ and $\P$ are monotone.
  \item[\rm{(t3)}] $x\leq \G\P x$ and $x\leq \H\F x$,
  \item[\rm{(t7)}] $\P\G x\leq x$ and $\F\H x\leq x$, 
  \end{itemize} 
\end{itemize}
\end{lemma}

\

Next, we state the precise relationship  between \tdlat-algebras and some well-known tense structures. In particular, we show that the notion of \tdlat-algebra is a natural generalization of tense Boolean algebras (tense algebras).
\begin{remark}
\begin{itemize}
  \item[] 
  \item[$\bullet$] Recall that a {\em tense Boolean algebra} (or simply {\em tense algebra} is a structure $\langle B,\wedge,\vee,\neg,\G,\H,0,1\rangle$ such that $\langle B,\wedge,\vee,\neg,0,1\rangle$ is a Boolean algebra such that $\G$ and $\H$ verify (t1), (t2) and (t3). The operators $\F$ and $\P$ are defined by $\F:=\neg \G\neg x$ and $\P:=\neg \H\neg$. Sometimes it will be convenient for us to consider tense algebras as structures $\langle\mathcal{B},\G,\H,\F, \P \rangle$ where $\mathcal{B}$ is the underlying Boolean algebra and $\F$ and $\P$ are the defined operators, that is, we include the non-primitive operators in the signature.
  \item[$\bullet$] Every tense algebra verifies all the axioms (t1)-(t8) and therefore, every tense algebra is a \tdlat-algebra.
  \item[$\bullet$] If the underlying lattice of a given \tdlat-algebra $\mathcal{A}$ is a Boolean algebra, then $\mathcal{A}$ is a tense algebra. Indeed, nest we check that under these conditions it holds $\F:=\neg \G\neg$ and $\P:=\neg \H\neg$.
\end{itemize}  
\end{remark}

\noindent Given a \tdlat-algebra $\mathcal{A}$ we denote by $B(A)$ the set of Boolean elements of $\mathcal{A}$. 

\begin{lemma}\label{L214} Let $\mathcal{A}$ be a \tdlat-algebra. Then,  $\mathcal{B}(\mathcal{A})=\langle B(A),\wedge,\vee,\neg, \G, \H,0,1\rangle$ is a tense algebra.
\end{lemma}
\begin{proof} Let us show that $\G x,\H x\in B(A)$, for every $x\in B(A)$. Indeed, if $x\in B(A)$ then there exists $\neg x$ such that $x\wedge\neg x=0$ and $x\vee\neg x=1$.\\
From (t1) and (t4) $1=\G 1 =\G (x\vee \neg x)\leq \G x \vee \F \neg x$ and from (t8) and (t5) $\G x \wedge \F \neg x \leq \F ( x\wedge\neg x)=\F 0=0$, and therefore $\F \neg x$ is the Boolean complement of $\G x $. Hence $\G x \in B(A)$. Similarly, it can be proved that $\G \neg x $ is the Boolean complement of $\F x $, that is, $\F x =\neg \G \neg x$. 

The proof that $\H x \in B(A)$ and $\P x =\neg \H \neg x$ is analogous.
\end{proof}

\begin{corollary} If $\mathcal{A}$ is a $\tdlat$-algebra such that $\mathcal{A}_0$ is a Boolean algebra then $\mathcal{A}$ is a tense algebra.
\end{corollary}

\begin{corollary} If $\mathcal{A}$ is a $\tdlat$-algebra such that $\mathcal{A}_0$ is a Boolean algebra then $\mathcal{A}^d=\langle A^{d}, \wedge, \vee, \neg, 0 , 1\rangle$ is a subalgebra of $\mathcal{A}_0$.
\end{corollary}

\begin{theorem}\label{T216} Let $\mathcal{B}$ be a Boolean algebra and let $ \G,\H,\F,\P$ be unary operators defined on $B$. Then, the following conditions are equivalent.
\begin{itemize}
  \item[(i)] $\langle\mathcal{B}, \G, \H, \F, \P\rangle$ is a tense algebra.
  \item[(ii)] $\langle\mathcal{B}, \G, \H, \F, \P\rangle$ is a \tdlat-algebra.
\end{itemize}
\end{theorem}

We denote by $\tdlat_{\bf B}$ the category whose objects are $\tdlat$-algebras with underlying lattice being a Boolean algebra and whose morphisms are the corresponding homomorphisms (in the sense of Universal Algebra).

\begin{remark}
    \begin{itemize}
        \item[]
        \item[$\bullet$] 
        Recall that an $IK_t$-algebra is a system $\langle \mathcal{A},\G,\H,\F,\P\rangle$ where $\mathcal{A}$ is a Heyting algebra and $\G,\H,\F$ and $\P$ satisfy the axioms of Definition \ref{D21} except (t4). These algebras are an algebraic counterpart of the logic $IK_t$ introduced by \cite{Ewald86} (see \cite{F2}). 
        \item[$\bullet$] Following the notation proposed by Edwald in \cite{Ewald86}, we call $IK_t^{\ast}$-algebra to any a $IK_t$-algebra which satisfies (t4).
    \end{itemize}
\end{remark}

\begin{theorem}\label{T218} Let $\mathcal{A}$ be a Heyting algebra and let $\G,\H,\F,\P$ be unary operators defined on $A$. Then, the following conditions are equivalent.
\begin{itemize}
  \item[(i)] $\langle\mathcal{A}, \G, \H, \F, \P\rangle$ is an $IKt^{\ast}$-algebra.
  \item[(ii)] $\langle\mathcal{A}, \G, \H, \F, \P\rangle$ is a \tdlat-algebra.
\end{itemize}
\end{theorem}

 Denote by $\tdlat_{\bf H}$ the category whose objects are  $\tdlat$-algebras whose underlying lattice are Heyting and whose morphisms are the corresponding homomorphisms.\\
 Recall that a {\em tense De Morgan algebra} is a system $\mathcal{A}=\langle A, \wedge, \vee, \sim, \G, \H, 0, 1\rangle$ such that $\langle A, \wedge, \vee, \sim, 0, 1\rangle$ is a De Morgan algebra and that satisfies conditions (t1)--(t4) where $\F x:={\sim}\G({\sim}x)$, $\P x:={\sim}\H({\sim} x)$ (see \cite{F6}).
 
\begin{theorem} \label{T219}Let $\mathcal{A}$ be a De Morgan algebra and let $ \G,\H,\F,\P$ be unary operators defined on $A$ such that $\F{:=\sim}\G{\sim}$ and $\P{:=\sim}\H{\sim}$ where $\sim$ is the De Morgan negation. Then, the following conditions are equivalent:
\begin{itemize}
  \item[(i)] $\langle\mathcal{A}, \G, \H,\F, \P\rangle$ is a tense De Morgan algebra.
  \item[(ii)] $\langle\mathcal{A}, \G, \H, \F, \P\rangle$ is a \tdlat-algebra.
\end{itemize}
\end{theorem}

We denote by $\tdlat_{\bf M}$ the category whose objects are $\tdlat$-algebras with underlying lattice being a De Morgan algebra and whose morphisms are the corresponding homomorphisms. 

\

\subsection{Prime filters in \tdlat-algebras}

Now, we initiate the study of prime filters in \tdlat-algebras and their properties.
Given a \tdlat-algebra $\mathcal{A}=\langle\mathcal{A}_0,\G,\H,\F,\P\rangle$ we denote by $X(\mathcal{A})$ the family of lattice filters of $\mathcal{A}_0$.\\
The following results are fundamental in the development of representation theorems for \tdlat-algebras.

\begin{lemma}\label{L220} Let $\mathcal{A}$ be a \tdlat-algebra and let $S, T$ two lattice filters of $\mathcal{A}_0$. The following conditions are equivalent. 
\begin{itemize}
  \item[(i)] $\G^{-1}(S)\subseteq T\subseteq \F^{-1}(S)$
  \item[(ii)] $\H^{-1}(T)\subseteq S\subseteq \P^{-1}(T)$
\end{itemize}
\end{lemma}
\begin{proof} Immediate from axioms (t3) and (t7).
\end{proof}

Some of the results that follow can be obtained from \cite{ChajdaPaseka}. For the sake of readability, we include them in this paper.

\begin{definition}\label{DRA} Let $\mathcal{A}$ be a \tdlat-algebra. Define on  $X(A)$ the binary relations $R_{\G\F}$ y $R_{\H\P}$ as follows.
$$(S,T)\in R_{\G\F}\,\,\,\mbox{ iff }\,\,\,\G^{-1}(S)\subseteq T\subseteq \F^{-1}(S)$$
$$(S,T)\in R_{\H\P}\,\,\,\mbox{ iff }\,\,\,\H^{-1}(S)\subseteq T\subseteq \P^{-1}(S)$$

\noindent where for every $\Phi\in \{\G,\H,\F,\P\}$, $\Phi^{-1}(S)=\{a\in A: \Phi(a)\in S\}$.
\end{definition}

\begin{proposition}\label{Pinv} $R_{\H\P}$ is the inverse relation $R_{\G\F}$.    
\end{proposition}
\begin{proof} Immediate from Lemma \ref{L220}.
\end{proof}

\begin{remark} Taking into account Proposition \ref{Pinv}, we denote simply by  $R_{A}$ the relation $R_{\G\F}$ and the the inverse relation is $R^{-1}_A=R_{\H\P}$
\end{remark}

\begin{lemma}\label{L224} Let $\A$ be a \tdlat-algebra. Then, for every $S\in X(A)$ it holds:
\begin{multicols}{2}
\begin{itemize} 
  \item[(i)] $R_A({\uparrow}S)\subseteq{\uparrow}R_A(S)$, 
  \item[(ii)] $R^{-1}_A({\uparrow}S)\subseteq{\uparrow}R^{-1}_A(S)$, 
  \item[(iii)] $R_A({\downarrow}S)\subseteq{\downarrow}R_A(S)$, 
  \item[(iv)] $R^{-1}_A({\downarrow}S)\subseteq{\downarrow}R^{-1}_A(S)$, 
  \item[(v)] $R_A(S)={\uparrow}R_A(S)\cap{\downarrow}R_A(S)$.
\end{itemize}
\end{multicols}
\end{lemma}
\begin{proof} We only prove (i), (iii) and (v) since the proof of (ii) and (iv) are similar to the one of (i) and (iii), respectively.\\
(i) \ Let $S, T$ two prime filters such that $T\in R_A({\uparrow}S)$, then there is a prime filter $Y$ such that 
$$(1)\,\,S\subseteq Y\hspace{1cm}\mbox{and}\hspace{1cm}(2)\,\,\G^{-1}(Y)\subseteq T\subseteq \F^{-1}(Y).$$
Now, consider the filter $\G^{-1}(S)$ and the ideal $I$ generated by $T^c\cup(\F^{-1}(S))^c$. Then, $\G^{-1}(S)\cap I=\emptyset$. Indeed, suppose that $(3)\,\,x\in \G^{-1}(S)$ such that $x\in I$. Then, there are elements  $(4)\,\,y\in T^c$ and $(5)\,\,z\in (\F^{-1}(S))^c$ such that $x\leq y\vee z$. Hence, from (t9) and (t4) we have $\G(x)\leq \G(y\vee z)\leq \G(y)\vee \F(z)$ and then, by (3) and the fact that $S$ is a prime filter,  $\G(y)\in S\,\,\,\mbox{o}\,\,\,\F(z)\in S$. From (5) we have $\F(z)\not\in S$ and then it holds $\G(y)\in S$. From (1) and (2), $y\in T$ which contradicts (4). By the Birkhoff-Stone Theorem, there is a prime filter $W$ such that
$$\G^{-1}(S)\subseteq  W\subseteq \F^{-1}(S)\,\,\,\mbox{y}\,\,\,W\subseteq T$$
That is, $T\in {\uparrow}R_A(S)$. \\[1.5mm]
(iii) \ Let $S, T$ be two prime filters such that  $T\in R_A({\downarrow}S)$. Then, there is a prime filter $Y$ such that 
$$(6)\,\,Y\subseteq S\hspace{1cm}\mbox{and}\hspace{1cm}(7)\,\,\G^{-1}(Y)\subseteq T\subseteq \F^{-1}(Y).$$
Let  $Z$ be the filter generated by  $\G^{-1}(S)\cup T$ and let us show that $Z\subseteq \F^{-1}(S)$. Indeed, let $x\in Z$, then there are $y\in \G^{-1}(S)$, $z\in T$ such that $y\wedge z\leq x$. Since $z\in T$, by (7) and (6) we have $\F(z)\in S$ and then $\G(y)\wedge \F(z)\in S$. Besides, by (t8) and (t9), $\G(y)\wedge \F(z)\leq \F(y\wedge z)\leq \F(x)$ and therefore $x\in \F^{-1}(S)$. Hence $Z\cap (\F^{-1}(S))^c=\emptyset$ and since $S$ is a prime filter, it is not difficult to see that $(\F^{-1}(S))^c$ is an ideal. Again, by  Birkhoff-Stone Theorem, there is a prime filter $W$ such that
$$\G^{-1}(S)\subseteq  W\subseteq \F^{-1}(S)\,\,\,\mbox{and}\,\,\,T\subseteq W$$
Than is,  $T\in {\downarrow}R_A(S)$. \\[1.5mm]
(v) \ It is immediate that $R_A(S)\subseteq{\uparrow}R_A(S)\cap{\downarrow}R_A(S)$. Conversely, let $T$ be a prime filter such that $T\in {\uparrow}R_A(S)$ and $T\in {\downarrow}R_A(S)$. Then, there are prime filters $Y, Z$ such that
$(8)\,\,S\subseteq Y$ and $(9)\,\,\G^{-1}(Y)\subseteq T\subseteq \F^{-1}(Y)$, besides
$(9)\,\,Z\subseteq S$ and $(10)\,\,\G^{-1}(Z)\subseteq T\subseteq \F^{-1}(Z)$. Then, $$\G^{-1}(S)\subseteq_{(8)} \G^{-1}(Y)\subseteq_{(9)} T\subseteq_{(10)} \F^{-1}(Z)\subseteq_{(9)} \F^{-1}(S).$$
That is,  $T\in R_A(S)$.
\end{proof}

\

\begin{lemma}\label{L225} Let $\mathcal{A}$ be a \tdlat-algebra, $S\in X(A)$ and $a\in A$. Then,
\begin{itemize}
  \item[(i)] $\G(a)\not\in S$ \ iff \ there is $T\in X(A)$ such that $T\in R_{A}(S)$ and $a\not\in T$. 
  \item[(ii)] $\H(a)\not\in S$ \ iff \ there is $T\in X(A)$ such that $T\in R^{-1}_{A}(S)$ and $a\not\in T$.
  \item[(iii)] $\F(a)\in S$ \ iff \ there is $T\in X(A)$ such that $T\in R_{A}(S)$ and $a\in T$.
  \item[(iv)] $\P(a)\in S$ \ iff \ there is $T\in X(A)$ such that $T\in R^{-1}_{A}(S)$ and $a\in T$.
\end{itemize}
\end{lemma}
\begin{proof} We only show  (i) and (iii), the rest are analogous.\\
(i) \ Suppose that $\G(a)\not\in S$ and check that (1) $\G^{-1}(S)\cap I=\emptyset$ where $I$ is the ideal generated by $\{a\}\cup(\F^{-1}(S))^c$. Suppose that (1) does not hold, then there is $x\in \G^{-1}(S)$ y $z\in (\F^{-1}(S))^c$ such that $x\leq a\vee z$.From (t9) and (t4), we have $\G(x)\leq \G(a\vee z)\leq \G(a)\vee \F(z)$ and since $S$ is a prime filter, we have that either $\G(a)\in S$ or $\F(z)\in S$, which is a contradiction. By the Birkhoff-Stone Theorem, there is a prime filter $T$ such that $\G^{-1}(S)\subseteq T\subseteq \F^{-1}(S)$ and $a\not\in T$.\\[1.5mm]
(iii) \ Suppose that $\F(a)\in S$ and let $Z$ be the filter generated by $\G^{-1}(S)\cup \{a\}$. Then, $(1)\,Z\cap(\F^{-1}(S))^c=\emptyset$. Indeed, suppose that there is $(2)\,x\in Z$ such that $(3)\,x\not\in \F^{-1}(S)$. From (2), there is  $(4)\,y\in \G^{-1}(S)$ such that $(5)\,y\wedge a\leq x$. By the hypothesis and (4), $\G(y)\wedge \F(a)\in S$, besides, from (5), (t8) and (t9) we have $\G(y)\wedge \F(a)\leq \F(y\wedge a)\leq \F(x)$. Then $\F(x)\in S$, which contradicts (3). From the Birkhoff-Stone Theorem, there is a prime filter $T$ such that
$\G^{-1}(S)\subseteq T\subseteq \F^{-1}(S)$ and $a\in T$.
\end{proof}

\

\subsection{Tense filters and ideals in \tdlat-algebras}

In order to obtain characterizations of important particular congruences of a given \tdlat-algebra, we introduce the notions of {\em tense filter} and {\em tense ideals} which we also call \tdlat-filters y \tdlat-ideals, respectively. In what follows, $\mathcal{A}=\langle\mathcal{A}_0,G,H,F,P\rangle$ is a \tdlat-algebra.

\begin{definition}[\tdlat-Filtro] A given lattice filter $S$ of  $\mathcal{A}_0$ is said to be a tense filter (or \tdlat-filter) if:
\begin{itemize}
  \item[(FLt)] $\G(x), \H(x)\in S$, for every $x\in S$.
\end{itemize}
\end{definition}

\noindent From Definition \ref{Dd}, one can easily check the following result. 

\begin{lemma}\label{L227} Let $S$ be a filter of $\mathcal{A}_0$. The following conditions are equivalent.
\begin{itemize}
  \item[(i)] $S$ is a $Lt$-filter.
  \item[(ii)] $d x\in S$, for every $x\in S$.
  \item[(iii)] $d^n x\in S$, for every $x\in S$ and every $n\in \omega$.
\end{itemize}
\end{lemma}

\noindent Let $X\subseteq A$ a non-empty set, we denote by $\mathcal{F}(X)$ the filter generated by $X$ in $\mathcal{A}_0$, that is:

$$\mathcal{F}(X)=\left\{y\in A: \,\,\mbox{ there is }\,\,n\in\mathbb{N}\,\,\mbox{and}\,\,\{x_i\}^n_{_{i=1}}\subseteq X,\,\,\mbox{ such that }\,\,\bigwedge^n_{i=1}x_i\leq y\right\}$$

\noindent and denote by $\langle X\rangle$ the \tdlat-filtro generated by $X$ in $\mathcal{A}$.

\begin{definition}\label{D228} Let $X\subseteq A$ a non-empty set. For every $p\in \omega$, define the set $D_p(X)$ as follows:
$$D_p(X):=\left\{y\in A:\,\,\mbox{ there is }\,\,n\in\mathbb{N}\,\,\mbox{ and }\,\,\{x_i\}^n_{_{i=1}}\subseteq X,\,\,\mbox{ such that }\,\,d^p\left(\bigwedge^n_{i=1}x_i\right)\leq y\right\}$$
\end{definition}

\begin{theorem}\label{T229} Let $X\subseteq A$ a non-empty set and let $\mathcal{D}=\left\{D_p(X)\right\}_{p\in\omega}$. Then, it holds:
\begin{itemize}
  \item[(1)] $\mathcal{D}$ is an upward chain of filters of $\mathcal{A}_0$.
  \item[(2)] $\bigcup\mathcal{D}$ is a \tdlat-filter of $\mathcal{A}$.
  \item[(3)] $\langle X\rangle=\bigcup\mathcal{D}$
\end{itemize}
\end{theorem}
\begin{proof} 
\begin{itemize}
  \item[(1)] For every $p\in\omega$, consider the set $d^p X=\{d^p x:x\in X\}$. By $(d_3)$, $d^p$ respect infinite infimum and then it is immediate that $D_p(X)=\mathcal{F}(d^p X)$. Therefore, $\mathcal{D}$ is a family of filters of  $\mathcal{A}_0$. Besides, from $(d_2)$ we have that $D_p(X)\subseteq D_{p+1}(X)$ for every $p\in \omega$ and then $\mathcal{D}$ is an upward chain.
  \item[(2)] By (1) we have that $\bigcup\mathcal{D}$ is a filter of $\mathcal{A}_0$. Then $d x\in \bigcup\mathcal{D}$ for every $x\in \bigcup\mathcal{D}$. Indeed, let $x\in \bigcup\mathcal{D}$ then there is $p\in \omega$ such that $x\in D_p(X)$, that is, there are $\{a_i\}^n_{_{i=1}}\subseteq X$ such that $$\displaystyle d^p\left(\bigwedge^n_{i=1}a_i\right)\leq x,$$ 
  from $(d_4)$ we have $\displaystyle d^{p+1}\left(\bigwedge^n_{i=1}a_i\right)\leq d(x)$ and then $d x\in D_{p+1}(X)$. Therefore, $d x\in \bigcup\mathcal{D}$ and by Lemma \ref{L227} we conclude that $\bigcup\mathcal{D}$ is a \tdlat-filter.
  \item[(3)] It is clear that $X\subseteq \bigcup\mathcal{D}$. In fact, $X\subseteq D_p(X)$ for every $p\in\omega$. Let $S$ be a \tdlat-filter such that $X\subseteq S$. By Lemma \ref{L227} we have $d^p X\subseteq S$ for every  $p\in\omega$, and since $S$ is a particular filter of $\mathcal{A}_0$ we have $\mathcal{F}(d^p X)\subseteq S$ for every $p\in\omega$. Then $D_p X\subseteq S$ for every  $p\in\omega$ and then $\bigcup\mathcal{D}\subseteq S$.
\end{itemize}
\end{proof}

\

\begin{corollary}[Compactness]\label{C230} Let $X\subseteq A$ be a non-empty set. Then 
$$\langle X \rangle=\left\{y\in A:\,\,\mbox{ there is }\,\,n\in\mathbb{N}\,\,\mbox{ and }\,\,\{x_i\}^n_{_{i=1}}\subseteq X,\,\,\mbox{ such that 
 }\,\,d^p\left(\bigwedge^n_{i=1}x_i\right)\leq y,\,\,\mbox{ for some }\,\,p\in\omega\right\}$$
\end{corollary}

\noindent In particular,

\begin{corollary}\label{C231} Let  $S$ be a \tdlat-filter of $\A$ and $a\in A$. Then 
$$\langle S\cup\{a\} \rangle=\{y\in A: x\wedge d^{p} a\leq y,\,\,\mbox{ for some }\,\,n\in\omega,x\in S\}$$
Besides, if  $a\in A^d$
$$\langle S\cup\{a\} \rangle=\{y\in A: x\wedge a\leq y,\,\,\mbox{ for some }\,\,x\in S\}=\mathcal{F}(S\cup\{a\})$$

\end{corollary}

\

\begin{corollary}\label{C232} If $a\in A$ then $$\langle a\rangle=\{x\in A: d^{p} a\leq x,\,\, \textnormal{ for some }\,\, p\in\omega\}.$$
Besides, if $a\in A^d$ then 
$$\langle a\rangle=\{x\in A: a\leq x\}=\mathcal{F}(a)$$
More over,  
\begin{center}
    $\langle a\rangle=\mathcal{F}(a)$  iff  $a\in A^d$
\end{center}
\end{corollary}

\

\begin{corollary}\label{C233} If $\mathcal{A}$ is finite then, the lattice $\mathcal{A}^d$ of  $d$-invariant elements of $\mathcal{A}$ is isomorphic to the dual lattice of  \tdlat-filters of $\mathcal{A}$.  
\end{corollary}

\noindent It is possible to establish the notion of \tdlat-ideal and prove the corresponding results. 

\begin{definition}[\tdlat-Ideal]\label{D234} A lattice ideal $I$ of $\mathcal{A}_0$ is said to be a \tdlat-ideal if 
\begin{itemize}
  \item[(ILt)] $\F(x), \P(x)\in I$, for every $x\in I$.
\end{itemize}
\end{definition}

Then, we can recast Lemma \ref{L227} in terms of the notion of \tdlat-ideal and using the map $\hat{d}$ instead of $d$. 

\begin{lemma}\label{L235} Let $I$ be a ideal of $\mathcal{A}_0$. The following conditions are equivalent.
\begin{itemize}
  \item[(i)] $I$ is a \tdlat-ideal.
  \item[(ii)] $\hat{d} x\in S$, for every $x\in S$.
  \item[(iii)] $\hat{d}^n x\in S$, for every $x\in S$ and every $n\in \omega$.
\end{itemize}
\end{lemma}

Besides, if we denote by $\langle X\rangle^{\partial}$ the \tdlat-ideal generated by $X\neq\emptyset$ in $\mathcal{A}$. Then, we can prove the following characterization. 

\begin{lemma}\label{L236} Let $X\subseteq A$ be a non-empty set. Then 
$$\langle X \rangle^{\partial}=\left\{y\in A:\,\,\mbox{ there is }\,\,n\in\mathbb{N}\,\,\mbox{ and }\,\,\{x_i\}^n_{_{i=1}}\subseteq X,\,\,\mbox{ such that }\,\,y\leq \hat{d}^p\left(\bigvee^n_{i=1}x_i\right),\,\,\mbox{ for some }\,\,p\in\omega\right\}$$
\end{lemma}

\section{Tense Priestley spaces and duality}\label{s3}
As it is usual in Algebraic Logic, here we introduce the notion of {\em tense Priestley space} and develop a topological duality for tense distributive lattices. Denote by \ps\ the category whose objects are Priestley spaces and whose morphisms are the usual functions between Priestley spaces.

\begin{definition}\label{DGR} Let $X$ be a non-empty set, $R$ a binary relation on  $X$ and $R^{-1}$ the inverse relation of $R$. Define operators on $\mathcal{P}(X)$ as follows: for every $Y\subseteq X$:
\begin{equation}
\G_R(Y)=\{x\in x: R(x)\subseteq Y\}
\end{equation}
\begin{equation}
\F_R(Y)=\{x\in X: R(x)\cap Y\neq\emptyset\}
\end{equation}
\begin{equation}
\H_{R^{-1}}(Y)=\{x\in X: R^{-1}(x)\subseteq Y\}
\end{equation}
\begin{equation}
\P_{R^{-1}}(Y)=\{x\in X: R^{-1}(x)\cap Y\neq\emptyset\}
\end{equation}
\end{definition}

\begin{proposition}\label{P32} If $X$ is a non-empty set and $R$ a binary relation on $X$ then
$$\langle X, R\rangle^+:=\langle \mathcal{P}(X),\cap, \cup, \G_R, \H_{R^{-1}}, \F_R, \P_{R^{-1}},\emptyset, X\rangle$$ 
is a \tdlat-algebra.
\end{proposition}
\begin{proof} It is clear that $\langle \mathcal{P}(X),\cap, \cup, \emptyset, X\rangle$ is a complete bounded distributive lattice. Besides, it is routine to check that $\G_R, \H_{R^{-1}}, \F_R, \P_{R^{-1}}$ verifies (t1)-(t8). 
\end{proof}

\begin{remark}\label{O33} If we consider the set complement on the structure of Proposition \ref{P32} we have that $$\langle X, R\rangle_{B}^+:=\langle \mathcal{P}(X),\cap, \cup, ^c,\G_R, \H_{R^{-1}}, \F_R, \P_{R^{-1}},\emptyset, X\rangle$$ 
is a tense algebra.    
\end{remark}
 
\begin{definition}[\tps-space]\label{DLTS} A tense Priestley space (or \tps-espace) is a pair $\mathfrak{X}=(X, R)$ where $X$ es a Priestley space and $R$ is a binary relation on $X$ such that:
\begin{itemize}
  \item[\rm{(\tps 1)}] for every $x\in X$, $R(x)$ is a closed subset of $X$,
  \item[\rm{(\tps 2)}] for every $x\in X$, $R(x)={\uparrow}R(x)\cap{\downarrow}R(x)$,
  \item[\rm{(\tps 3)}] $\G_R(U), \H_{R^{-1}}(U), \F_R(U), \P_{R^{-1}}(U)\in D(X)$, for every $U\in D(X)$,
\end{itemize} 
where $G_R, H_{R^{-1}}, F_R$ and $P_{R^{-1}}$ are as in Definition \ref{DGR}, and  $D(X)$ is the set of all upward clopens of $X$.
\end{definition}

\begin{definition}[\tps-function]\label{DLTF} A \tps-function from the \tps-space $\mathfrak{X}_1$ on the \tps-space $\mathfrak{X}_2$ is a continuous function  $f:X_1\to X_2$ that preserves the order and satisfies the following conditions.
\begin{itemize}
  \item[\rm{(\tps f1)}] $f(R_1(x))\subseteq R_2(f(x))$, for every $x\in X_1$.
  \item[\rm{(\tps f2)}] For every $y\in X_2$ and every $x\in X_1$,
  \item[] if $y\in R_2(f(x))$ then, there are $z,w\in R_1(x)$ such that  $f(z)\leq y\leq f(w)$.
  \item[\rm{(\tps f3)}] For every $y\in X_2$ and every $x\in X_1$,
  \item[] if  $y\in R^{-1}_2(f(x))$ then, there are $z,w\in R^{-1}_1(x)$ such that $f(z)\leq y\leq f(w)$.
\end{itemize}
\end{definition}

\

\begin{remark} Axiom {\rm(\tps f1)} is equivalent to: \, {\rm(\tps f1')} $f(R^{-1}_1(x))\subseteq R^{-1}_2(f(x))$, for every $x\in X_1$.
\end{remark}

\

\noindent We denote by $\tps$ the category whose objects are $\tps$-spaces and whose morphisms are \tps-functions.

\begin{lemma}\label{L37} Let $f$ be a \tps-function from $\mathfrak{X}_1$ in $\mathfrak{X}_2$. The following conditions are equivalent.
\begin{itemize}
  \item[(i)] f is an isomorphism in the category \tps.
  \item[(ii)] f es an isomorphism in the category \ps\ such that
  \item[] $f(R_1(x))=R_2(f(x))$, for all $x\in X_1$. 
\end{itemize}
\end{lemma}
\begin{proof} It is routine.
\end{proof}

\

From the Proposition \ref{P32} and axiom {\rm(\tps 3)} are immediate the following lemmas.

\begin{lemma}\label{L38} If $\mathfrak{X}$ is a \tps-space, then ${\bf \Psi}(\mathfrak{X})=\langle D(X),\cap, \cup, \G_R, \H_{R^{-1}}, \F_R, \P_{R^{-1}},\emptyset,X\rangle$ is a \tdlat-algebra. 
\end{lemma}

\begin{lemma}\label{L39} If $f$ is a  \tps-function from $\mathfrak{X}_1$ in $\mathfrak{X}_2$ and $U\in D(X_2)$, then for every $x\in X_1$ it holds:
\begin{itemize}
  \item[(i)] $R_2(f(x))\subseteq U$ \, iff \,  $f(R_1(x))\subseteq U$,
  \item[(ii)] $R^{-1}_2(f(x))\subseteq U$ \, iff \, $f(R^{-1}_1(x))\subseteq U$,
  \item[(iii)] $R_2(f(x))\cap U\not=\emptyset$ \, iff \, $R_1(x)\cap f^{-1}(U)\not=\emptyset$,
  \item[(iv)] $R^{-1}_2(f(x))\cap U\not=\emptyset$ \, iff \,  $R^{-1}_1(x)\cap f^{-1}(U)\not=\emptyset$.
\end{itemize}
\end{lemma}
\begin{proof} We only prove (i) and (iii) the rest are similar using (\tps f1') and (\tps f3).\\[1.5mm]
(i): $R_2(f(x))\subseteq U$ implies $f(R_1(x))\subseteq U$, immediate from (\tps f1).\\
On the other hand, suppose that $f(R_1(x))\subseteq U$ and take $y\in R_2(f(x))$ then, by (\tps f2), there is $z\in R_1(x)$ such that $f(z)\leq y$. Then, from the hypothesis we have $f(z)\in U$ and, since $U$ is a upward set, $y\in U$. Therefore,  $R_2(f(x))\subseteq U$.\\[2mm]
(iii): Suppose that $R_2(f(x))\cap U\not=\emptyset$. Then, there is $y\in R_2(f(x))$ such that $y\in U$. By (\tps f2), there is $w\in R_1(x)$ such that $y\leq f(w)$, and then, since $U$ as an upward set, we know that $f(w)\in U$, and therefore $w\in R_1(x)\cap f^{-1}(U)$. That is, $R_1(x)\cap f^{-1}(U)\not=\emptyset$.\\
Conversely, assume that $R_1(x)\cap f^{-1}(U)\not=\emptyset$. Taking into account (\tps f1),
$$\emptyset\not=f(R_1(x)\cap f^{-1}(U))\subseteq f(R_1(x))\cap f(f^{-1}(U))\subseteq R_2(f(x))\cap U.$$
\end{proof}

\begin{lemma}\label{L310} Let $f$ be a \tps-function from $\mathfrak{X}_1$ in $\mathfrak{X}_2$. Then,  ${\bf \Psi}(f):D(X_2)\longrightarrow D(X_1)$ defined by ${\bf \Psi}(f)(U)=f^{-1}(U)$ for every $U\in D(X_2)$, is a morphism in the category \tdlat.
\end{lemma}
\begin{proof} We know that ${\bf \Psi}(f):D(X_2)\longrightarrow D(X_1)$ is a morphism in the category \dl. We just need to show that it preserves all tense operators, or equivalently, he have to prove that for every $U\in D(X_2)$ it holds:
\begin{itemize}
  \item[{\bf(a)}] $f^{-1}(\G_{R_2}(U))=\G_{R_1}(f^{-1}(U))$,
  \item[{\bf(b)}] $f^{-1}(\H_{R^{-1}_2}(U))=\H_{R^{-1}_1}(f^{-1}(U))$,
  \item[{\bf(c)}] $f^{-1}(\F_{R_2}(U))=\F_{R_1}(f^{-1}(U))$,
  \item[{\bf(d)}] $f^{-1}(\P_{R^{-1}_2}(U))=\P_{R^{-1}_1}(f^{-1}(U))$.
\end{itemize}
Let $U\in D(X_2)$

\

\noindent{\bf(a)}  
\begin{eqnarray*}
x\in f^{-1}(\G_{R_2}(U)) & \mbox{iff} & f(x)\in \G_{R_2}(U)\\
                        & \mbox{iff} & R_2(f(x))\subseteq U\\
                        & \mbox{iff} & f(R_1(x))\subseteq U\hspace{2cm}\mbox{\mbox{Lemma  \ref{L39} (i)}}\\
                        & \mbox{iff} & R_1(x)\subseteq f^{-1}(U)\\
                        & \mbox{iff} & x\in \G_{R_1}(f^{-1}(U))
\end{eqnarray*}

\noindent{\bf(c)} 
\begin{eqnarray*}
x\in f^{-1}(\F_{R_2}(U)) & \mbox{iff} & f(x)\in \F_{R_2}(U)\\
                        & \mbox{iff} & R_2(f(x))\cap U\not=\emptyset\\
                        & \mbox{iff} & R_1(x)\cap f^{-1}(U)\not=\emptyset\hspace{2cm}\mbox{\mbox{Lemma  \ref{L39} (iii)}}\\
                        & \mbox{iff} & x\in \F_{R_1}(f^{-1}(U))\\
\end{eqnarray*}

\noindent{\bf(b)} and {\bf(d)} are proved analogously using Lemma \ref{L39} (ii) and (iv). 
\end{proof}

\

\begin{theorem} ${\bf \Psi}:\tps\longrightarrow \tdlat$ is a contravariant functor.
\end{theorem}

\

\begin{lemma}\label{L312} Let $\mathcal{A}$ be a  \tdlat-algebra. Then,  ${\bf \Phi}(\mathcal{A})=(X(A),R_A)$ is a \tps-space and $\sigma_A:A\longrightarrow D(X(A))$ defined by  $\sigma_A(a)=\{T\in X(A):\,a\in T\}$, for all $a\in A$, is an isomorphism in the category \tdlat.
\end{lemma}
\begin{proof} We know that $X(A)$ is a Priestley space and $\sigma_A:A\longrightarrow D(X(A))$ is an isomorphism in the category \dl.

Let us show that $\sigma_A$ is an isomorphism in \tdlat-algebra. So, we only have to check that $\sigma_A$ preserves all tense operators, that is, for every $x\in A$ it holds

\begin{itemize}
  \item[{\bf(a)}] $\sigma_A(\G(x))=\G_{R_A}(\sigma_A(x))$,
  \item[{\bf(b)}] $\sigma_A(\H(x))=\H_{R^{-1}_A}(\sigma_A(x))$,
  \item[{\bf(c)}] $\sigma_A(\F(x))=\F_{R_A}(\sigma_A(x))$,
  \item[{\bf(d)}] $\sigma_A(\P(x))=\P_{R^{-1}_A}(\sigma_A(x))$.
\end{itemize}
\begin{itemize}
  \item[] {\bf(a)} For $\G_{R_A}(\sigma_A(x))\subseteq \sigma_A(\G(x))$, take a prime filter $S$ such that $S\not\in \sigma_A(\G(x))$. Then, $\G(x)\not\in S$ and, by (i) and Lemma \ref{L225}, there is $T\in X(A)$ such that $T\in R_{A}(S)$ and $x\not\in T$. Then $R_{A}(S)\not\subseteq \sigma_A(x)$ and therefore $S\not\in \G_{R_A}(\sigma_A(x))$. Conversely, it easy to check $\sigma_A(\G(x))\subseteq \G_{R_A}(\sigma_A(x))$.
  \item[] {\bf(c)} Let $S\in \sigma_A(\F(x))$, then  $\F(x)\in S$ and by Lemma \ref{L225} (iii), there is $T\in X(A)$ such that $T\in R_{A}(S)$ and $x\in T$. Consequently, $R_{A}(S)\cap \sigma_A(x)\not=\emptyset$ and then $S\in \F_{R_A}(\sigma_A(x))$. Therefore $\sigma_A(\F(x))\subseteq \F_{R_A}(\sigma_A(x))$. The other inclusion is similar.
  \item[] {\bf(b)} and {\bf(d)} are similar using (ii) y (iv) del Lema \ref{L225}
\end{itemize}

To see that ${\bf \Phi}(\mathcal{A})$ is a \tdlat-space we just need to prove (\tps 1)-(\tps 3).

\begin{itemize}
  \item[(\tps 1)] $R_A(S)$ is a closed subset of $X(A)$ for all $S\in X(A)$. Indeed:
  \item[] Let $S\in X(A)$. Let us see that the complement of $R_A(S)$ is an open set. Let $T\in X(A)$ such that $T\not\in R_A(S)$, then either $(1)\,\,\,\G^{-1}(S)\not\subseteq T$ or $(2)\,\,\,T\not\subseteq \F^{-1}(S)$. If it holds (1) then there is $x\in \G^{-1}(S)$ such that $x\not\in T$, and then it is not difficult to check that $R_A(S)\subseteq \sigma_A(x)$. On the other hand, if it holds (2), then there is $y\in T$ such that $y\not\in \F^{-1}(S)$, and in this case we have  $R_A(S)\subseteq (\sigma_A(y))^c$. In both cases, we see that there is at least one basic in  $(R_A(S))^c$ and therefore $R_A(S)$ is closed. 
  \item[(\tps 2)] $R_A(S)={\uparrow}R_A(S)\cap{\downarrow}R_A(S)$, for all $S\in X(A)$ is consequence of  Lemma \ref{L224} (v).
  \item[(\tps 3)] $\G_{R_A}(U), \H_{R^{-1}_A}(U), \F_{R_A}(U), \P_{R^{-1}_A}(U)\in D(X(A))$, for all $U\in D(X(A))$. Indeed:
  \item[] Let $U\in D(X(A))$. Since $\sigma_A$ is an onto map, there is $x\in A$ such that $\sigma_A(x)=U$ and, taking into account {\bf(a)}, we have $\G_{R_A}(U)=\G_{R_A}(\sigma_A(x))=\sigma_A(\G(x))\in D(X(A))$. The rest are proves analogously using {\bf(b)}, {\bf(c)} and {\bf(d)}.
\end{itemize}
\end{proof}

\

\begin{lemma}\label{L313} Let $h:\mathcal{A}\longrightarrow \mathcal{B}$ a morphism in the category \tdlat. Then, ${\bf \Phi}(h):(X(B),R_{B})\to (X(A),R_{A})$ defined by ${\bf \Phi}(h)(S)=h^{-1}(S)$ for all $S\in X(B)$, is a morphism in the category \tps.
\end{lemma}
\begin{proof} We know that ${\bf \Phi}(h):X(B)\longrightarrow X(A)$ is a continuous function that preserves the order. We just need to show (\tps f1)-(\tps f3)
\begin{itemize}
  \item[(\tps f1)] ${\bf \Phi}(h)(R_B(S))\subseteq R_A({\bf \Phi}(h)(S))$, for every $S\in X(B)$. Indeed:
  \item[] Let $S\in X(B)$ and $T\in {\bf \Phi}(h)(R_B(S))$. Then, there is $T_0\in R_B(S)$ such that ${\bf \Phi}(h)(T_0)=T$, that is $\G^{-1}_B(S)\subseteq T_0\subseteq \F^{-1}_B(S)$. Then $h^{-1}(\G^{-1}_B(S))\subseteq h^{-1}(T_0)\subseteq h^{-1}(\F^{-1}_B(S))$, or equivalently $$(\G_B\circ h)^{-1}(S)\subseteq h^{-1}(T_0)\subseteq (\F_B\circ h)^{-1}(S).$$
   Since $h$ is a homomorphism between \tdlat-algebras,  $h\circ \G_A=\G_B\circ h$ and $h\circ \F_A=\F_B\circ h$. Then $$(h\circ \G_A)^{-1}(S)\subseteq h^{-1}(T_0)\subseteq (h\circ \F_A)^{-1}(S),$$
  that is, $\G^{-1}_A(h^{-1}(S))\subseteq h^{-1}(T_0)\subseteq \F^{-1}_A(h^{-1}(S))$ and therefore $(h^{-1}(S),h^{-1}(T_0))\in R_A$. Then, $h^{-1}(T_0)\in R_A(h^{-1}(S))$, that is,  $T\in R_A({\bf \Phi}(h)(S))$.
  \item[(\tps f2)] For every $T\in X(A)$ and every $S\in X(B)$,
  \item[] if $T\in R_{A}({\bf \Phi}(h)(S))$, then there are $Z,W\in R_{B}(S)$ such that ${\bf \Phi}(h)(Z)\subseteq T\subseteq {\bf \Phi}(h)(W)$.
  \item[(\tps f3)] For every $T\in X(A)$ and every $S\in X(B)$,
  \item[] if $T\in R^{-1}_{A}({\bf \Phi}(h)(S))$, then there are $Z,W\in R^{-1}_{B}(S)$ such that ${\bf \Phi}(h)(Z)\subseteq T\subseteq {\bf \Phi}(h)(W)$.
\end{itemize}
We only prove (\tps f2). The proof of (\tps f3) is analogous.
Let $T\in X(A)$ and $S\in X(B)$ such that $T\in R_{A}({\bf \Phi}(h)(S))$, that is  $(h^{-1}(S),T)\in R_{A}$ and therefore $$(1)\,\,\G_A^{-1}(h^{-1}(S))\subseteq T\,\,\,\mbox{ and }\,\,\,(2)\,\,T\subseteq \F_A^{-1}(h^{-1}(S)).$$
Let us see that 
\begin{itemize}
  \item[{\bf(a)}] there is $Z\in R_{B}(S)$ such that ${\bf \Phi}(h)(Z)\subseteq T$, and
  \item[{\bf(b)}] there is $W\in R_{B}(S)$ such that $T\subseteq {\bf \Phi}(h)(W)$.
\end{itemize}
Indeed, 
\begin{itemize}
  \item[{\bf(a)}] consider the filter $\G_B^{-1}(S)$ and the ideal $I$ generated by  $h(T^c)\cup (\F_B^{-1}(S))^c$. Then $$(\ast)\,\,\,\,G_B^{-1}(S)\cap I=\emptyset.$$
 Suppose that it is not the case $(\ast)$. Then, there are $x,y,z\in B$ such that $$(3)\,\,x\in \G_B^{-1}(S),\,\,(4)\,\,y\in h(T^c),\,\,(5)\,\,z\in (\F_B^{-1}(S))^c\,\,\,\mbox{ such that }\,\,(6)\,\,x\leq y\vee z.$$
  Then, from (6), (t9), (t4) and (3) we have $\G_B x \leq \G_B(y\vee z)\leq \G_B y \vee \F_B z \in S$. From (5) and the fact that $S$ is a prime filter, we have  $(7)\,\,\G_B y \in S$. Besides, by (4), there is $w\in T^c$ such that $h(w)=y$ and then $\G_B y =\G_B h(w) =h(\G_A w )$. From (7),  $h(\G_A w )\in S$  and then  $w\in \G_A^{-1}(h^{-1}(S))$. Taking into account (1) we have $w\in T$ which is a contradiction. Therefore, $(\ast)$ holds. By the Birkhoff-Stone Theorem, there is  $Z\in X(B)$ such that $\G_B^{-1}(S)\subseteq Z$ and $Z\cap I=\emptyset$. By the definition of $I$, $Z\cap(\F_B^{-1}(S))^c =\emptyset$ and $Z\cap h(T^c)=\emptyset$.\
  
  From the first condition, we have $\G_B^{-1}(S)\subseteq Z\subseteq \F_B^{-1}(S)$, that is $Z\in R_B(S)$. Besides, since $h^{-1}(Z)\cap T^c\subseteq h^{-1}(Z)\cap h^{-1}(h(T^c))=h^{-1}(Z\cap h(T^c))$. The second condition implies $h^{-1}(Z)\subseteq T$, that is, ${\bf \Phi}(h)(Z)\subseteq T$.
  \item[{\bf(b)}] Let $Q$ be the filter generated by $\G_B^{-1}(S)\cup h(T)$ and prove that  $Q\subseteq \F_B^{-1}(S)$.
  \item[] Notice that $h(T)$ not necessarily is a filter, but it is always closed by  $\wedge$ and this is enough to characterize  $Q$ as follows: $x\in Q$ \, iff \, there are $y\in G_B^{-1}(S), z\in h(T)$ such that $y\wedge z\leq x$.
  \item[] Let $x\in Q$. Then, there are $y,z\in B$ such that  $$\mbox{(8) }\,\,y\in \G_2^{-1}(S),\,\,\,\mbox{(9) }\,\,z\in h(T)\,\,\mbox{ and }\,\,\,\mbox{(10) }\,\,y\wedge z\leq x$$ 
  From (9), there is $w\in T$ such that $h(w)=z$. Then, by (2), $w\in \F_A^{-1}(h^{-1}(S))$ from where $h(\F_A w )=\F_B h(w) =\F_B z \in S$, and then, from (8) we have $\G_B y \wedge \F_B z \in S$. Besides, by (10), (t9) and (t8), $\G_B y \wedge \F_B z \leq \F_B(y\wedge z)\leq \F_B x $ and then $\F_B x \in S$. Therefore, $Q\subseteq \F_B^{-1}(S)$.
  From the above we know  $Q\cap(\F_B^{-1}(S))^c=\emptyset$. Then, by the Birkhoff-Stone Theorem, there is $W\in X(B)$ such that $Q\subseteq W$ and $W\cap (\F_B^{-1}(S))^c=\emptyset$. From this, $W\subseteq \F_B^{-1}(S)$. Besides, by the definition of $Q$, we have $\G_B^{-1}(S)\subseteq Q$ and $h(T)\subseteq Q$. Then, $\G_B^{-1}(S)\subseteq W\subseteq \F_B^{-1}(S)$ and $T\subseteq h^{-1}(h(T))\subseteq h^{-1}(W)$, that is, $W\in R_B(S)$ y $T\subseteq {\bf \Phi}(h)(W)$.
\end{itemize}
\end{proof}
    
\begin{theorem} ${\bf \Phi}:\tdlat\longrightarrow\tps$ is a contravariant functor.
\end{theorem}

\

\begin{lemma}\label{L315} Let $\mathfrak{X}$ be a \tps-space, then $\varepsilon_{X}:X\longrightarrow X(D(X))$ defined by $\varepsilon_{X}(x)=\{U\in  D(X): x\in U\}$, is an isomorphism in the category \tps.
\end{lemma}
\begin{proof} We know that $\varepsilon_{X}$ is an isomorphism in the category \ps, so, by Lemma \ref{L37}, we just have to prove that for every $x\in X$, $\varepsilon_{X}(R(x))=R_{_{D(X)}}(\varepsilon_{X}(x))$ or equivalently
$$(*)\,\,(x,y)\in R\,\,\,\mbox{ iff }\,\,\,(\varepsilon_{X}(x),\varepsilon_{X}(y))\in R_{_{D(X)}}.$$
Recall that the relation $R_{_{D(X)}}$ is defined as follows. Given $U,V\in X(D(X))$
$$(U,V)\in R_{_{D(X)}}\,\,\,\mbox{ iff }\,\,\,\G^{-1}_R(U)\subseteq V\subseteq \F^{-1}_R(U)$$
First, notice that the following conditions are equivalent.

\begin{itemize}
  \item[(1)] $(\varepsilon_{X}(x),\varepsilon_{X}(y))\in R_{_{D(X)}}$
  \item[(2)] $\G^{-1}_R(\varepsilon_{X}(x))\subseteq \varepsilon_{X}(y)\subseteq \F^{-1}_R(\varepsilon_{X}(x))$ 
  \item[(3)] For every $U,V\in D(X)$ it holds
$$(\G_R(U)\in \varepsilon_{X}(x)\,\,\mbox{ implies }\,\,U\in \varepsilon_{X}(y))\,\,\,\,\&\,\,\,\,(\,V\in \varepsilon_{X}(y)\,\,\mbox{ implies } \,\,\F_R(U)\in \varepsilon_{X}(x)).$$ 
  \item[(4)] For every $U,V\in D(X)$ it holds
$$(\,x\in \G_R(U)\,\,\mbox{ implies } \,\,y\in U)\,\,\,\,\&\,\,\,\,(\,y\in V\,\,\mbox{ implies } \,\,x\in \F_R(V)).$$ 
  \item[(5)] For every $U,V\in D(X)$ it holds
 $$(R(x)\subseteq U\,\,\,\mbox{ implies }\,\,\,y\in U)\,\,\&\,\,(y\in V\,\,\,\mbox{ implies }\,\,\,R(x)\cap V\not=\emptyset)$$ 
\end{itemize}

Then,
\begin{itemize}
  \item[{\bf(a)}] if $(x,y)\in R$ then $(\varepsilon_{X}(x),\varepsilon_{X}(y))\in R_{_{D(X)}}$. Indeed, 
  \item[] suppose that $(x,y)\in R$, that is $y\in R(x)$. Then, we have that, for every $U,V\in X(D(X))$, if $R(x)\subseteq U$ then $y\in U$. Besides, if $y\in V$ then $R(x)\cap V\not=\emptyset$ and therefore $(\varepsilon_{X}(x),\varepsilon_{X}(y))\in R_{_{D(X)}}$.
  \item[{\bf(b)}] If $(\varepsilon_{X}(x),\varepsilon_{X}(y))\in R_{_{D(X)}}$, then $(x,y)\in R$. Indeed, 
  \item[] consider $(\varepsilon_{X}(x),\varepsilon_{X}(y))\in R_{_{D(X)}}$ and suppose that $y\not\in R(x)$. By (\tps 2), $R(x)={\uparrow}R(x)\cap{\downarrow}R(x)$. We have the following two cases:
  \item[] {\bf Case 1:} $y\not\in {\uparrow}R(x)$. Then, for each $z\in R(x)$, $z\not\leq y$. Since $X$ is a totally disconnected space in the order, for each $z\in R(x)$ there is $U_z\in D(X)$ such that $z\in U_z$ and $y\not\in U_z$, from where we know  $R(x)\subseteq\displaystyle\bigcup_{z\in R(x)}U_z$ and $y\not\in \displaystyle\bigcup_{z\in R(x)}U_z$. Besides, since $X$ is compact, by (\tps 1) we have $R(x)$ is compact and then there are $z_1,\cdots,z_n\in R(x)$ such that $R(x)\subseteq\displaystyle\bigcup_{i=1}^{n}U_{z_i}$ and $y\not\in \displaystyle\bigcup_{i=1}^{n}U_{z_i}$. Then, there is  $U:=\displaystyle\bigcup_{i=1}^{n}U_{z_i}\in D(X)$ such that $R(x)\subseteq U$ and $y\not\in U$ which is a contradiction.
  \item[] {\bf Case 2:} $y\not\in {\downarrow}R(x)$. Then, for each $z\in R(x)$, $y\not\leq z$, then for each $z\in R(x)$ there is $V_z\in D(X)$ such that $y\in V_z$ and $z\not\in V_z$. Then, $R(x)\subseteq \displaystyle\bigcup_{z\in R(x)}(V_z^c)$. Using the same argument as in case 1, there are  $z_1,\cdots,z_n\in R(x)$ such that $R(x)\subseteq\displaystyle\bigcup_{i=1}^{n}(V_{z_i}^c)=\displaystyle\left(\bigcap_{i=1}^{n}V_{z_i}\right)^c$. Then, $V:=\displaystyle\bigcap_{i=1}^{n}V_{z_i}\in D(X)$ and it holds that $y\in V$ and $R(x)\cap V=\emptyset$ which is a contradiction.
\end{itemize}
Then, $(\ast)$ holds and, by Lemma \ref{L37},  $\varepsilon_X$ is an isomorphism in the category \tps.
\end{proof}

\

\noindent Now, we can prove the main result of this section.

\begin{theorem}\label{TDT} The categories \tps\ and \tdlat\ are naturally equivalent.
\end{theorem}
\begin{proof} Consider the functors ${\bf \Psi}:\tps\longrightarrow\tdlat$ and ${\bf \Phi}:\tdlat\longrightarrow\tps$. From Lemma \ref{L312}, we have that $\{\sigma_A: A\in  \tdlat\}$ is a family of isomorphisms in \tdlat. Besides, for every $A,A'\in \tdlat$ and  $h:A\longrightarrow A'$ the following diagram commutes

\vspace{2cm}
\begin{figure}[htbp]
\begin{center}
\hspace{0.25cm}
\begin{picture}(0,0)(0,0)
\put(-50,40){\makebox(1,1){$A$}}
\put(55,40){\makebox(1,1){${\bf \Psi}({\bf \Phi}(A))$}}
\put(-50,00){\makebox(1,1){$A'$}}
\put(55,00){\makebox(1,1){${\bf \Psi}({\bf \Phi}(A'))$}}
\put(-40,40){\vector(1,0){70}}
\put(-40,00){\vector(1,0){70}}
\put(-50,33){\vector(0,-1){25}}
\put(50,33){\vector(0,-1){25}}
\put(00,-10){\makebox(2,2){$\sigma_{A'}$}}
\put(00,45){\makebox(2,2){$\sigma_{A}$}}
\put(-60,20){\makebox(2,2){${\tiny h}$}}
\put(75,20){\makebox(2,2){${\bf \Psi}({\bf \Phi}(h))$}}
\end{picture}
\end{center}
\end{figure}

\noindent that is $({\bf \Psi}\circ{\bf \Phi})(h)\circ \sigma_A=\sigma_{A'}\circ Id_{\tdlat}(h)$, from where ${\bf \Psi}\circ{\bf \Phi}$ is naturally equivalent $Id_{\tdlat}$.\\
On the other hand, by Lemma \ref{L315}, $\{\varepsilon_X: X\in \tps\}$ is a family of isomorphisms in \tps\, such that for every $(X,R),(X',R')\in \tps$ and every morphism $f:X\longrightarrow X'$ the following diagram commutes

\vspace{2cm}
\begin{figure}[htbp]
\begin{center}
\hspace{0.25cm}
\begin{picture}(0,0)(0,0)
\put(-50,40){\makebox(1,1){$X$}}
\put(55,40){\makebox(1,1){${\bf \Phi}({\bf \Psi}(X))$}}
\put(-50,00){\makebox(1,1){$X'$}}
\put(55,00){\makebox(1,1){${\bf \Phi}({\bf \Psi}(X'))$}}
\put(-40,40){\vector(1,0){70}}
\put(-40,00){\vector(1,0){70}}
\put(-50,33){\vector(0,-1){25}}
\put(50,33){\vector(0,-1){25}}
\put(00,-10){\makebox(2,2){$\varepsilon_{X'}$}}
\put(00,45){\makebox(2,2){$\varepsilon_X$}}
\put(-60,20){\makebox(2,2){${\tiny f}$}}
\put(75,20){\makebox(2,2){${\bf \Phi}({\bf \Psi}(f))$}}
\end{picture}
\end{center}
\end{figure}

\noindent that is, $({\bf \Phi}\circ{\bf \Psi})(f)\circ \varepsilon_X=\varepsilon_{X'}\circ Id_{\tps}(f)$ and therefore ${\bf \Phi}\circ{\bf \Psi}$ is naturally equivalent to $Id_{\tps}$.
\end{proof}

\
\section{Some applications of the topological duality} \label{s4}
In this section, we apply the equivalence proved in Section \ref{s3} to obtain 
nice characterizations of important algebraic notions.\\
Firstly, we describe certain sets of a given \tps-space which play an important role in what follows.
\begin{definition}[\tps-set] \label{DTC} Let $\mathfrak{X}=(X,R)$ be a \tdlat-space. A subset $Y$ of $X$ is said to be a \tps-set of $\mathfrak{X}$ if it satisfies the following conditions: for every $x,y\in X$
\begin{itemize}
  \item[{\rm(tc1)}] if $x\in R^{-1}(y)\cap Y$ then there are $w_1,w_2\in R(x)\cap Y$ such that $w_1\leq y\leq w_2$,
  \item[{\rm(tc2)}] if $x\in R(y)\cap Y$ then there are $w_1,w_2\in R^{-1}(x)\cap Y$ such that $w_1\leq y\leq w_2$.  
\end{itemize}
\end{definition}

\noindent Denote by $C_{t}(X)$ the family of all closed \tps-sets of $\mathfrak{X}$. Upward and downward \tps-sets can be characterized as follows.

\begin{lemma}\label{L42} Let $\mathfrak{X}$ be a \tps-space and let $Y$ be an upward (downward) subset of $\mathfrak{X}$. The following conditions are equivalent.
\begin{itemize}
  \item[(i)] $Y$ is a \tps-set of $\mathfrak{X}$,
  \item[(ii)] for all $x\in Y$, the following conditions are satisfied
  \item[]{\rm(tc3)} $R(x)\subseteq Y$, 
  \item[]{\rm(tc4)} $R^{-1}(x)\subseteq Y$  ,
  \item[(iii)] $Y=G_R(Y)\cap Y\cap H_{R^{-1}}(Y)$,
  \item[(iv)] $Y=F_R(Y)\cup Y\cup P_{R^{-1}}(Y)$.
\end{itemize}
\end{lemma}
\begin{proof}
\noindent\emph{(i) implies (ii):} Let $x\in Y$ and $z\in R(x)$. Then, $x\in R^{-1}(z)\cap Y$ and, by (tc1), there are $w_1,w_2\in R(x)\cap Y$ such that $w_1\leq z\leq w_2$. If $Y$ is downward, since $w_1\in Y$ we have $z\in Y$; if $Y$ is downward, since $w_2\in Y$ we have $z\in Y$. In both cases, we conclude $R(x)\subseteq Y$. Similarly, by (tc2) we have $R^{-1}(x)\subseteq Y$.

\noindent\emph{(ii) implies (iii):} Let $x\in Y$. Then, $R(x)\subseteq Y$ and $R^{-1}(x)\subseteq Y$, that is, $x\in G_R(Y)\cap Y\cap H_{R^{-1}}(Y)$ and therefore it holds $(iii)$.

\noindent\emph{(iii) implies (iv):} Let $x\in F_R(Y)\cup Y\cup P_{R^{-1}}(Y)$. In 
 first place, suppose that $x\in F_R(Y)$. Then, $R(x)\cap Y\not=\emptyset$ and so there is $z\in R(x)$ such that  $z\in Y$. By the hypothesis, $z\in H_{R^{-1}}(Y)$ and then $R^{-1}(z)\subseteq Y$ and so $x\in Y$. Suppose now that $x\in P_{R^{-1}}(Y)$ then there is $z'\in R^{-1}(x)\cap Y$ and, by hypothesis, $z'\in G_R(Y)$, that is $R(z')\subseteq Y$ and then $x\in Y$. From all the above $F_R(Y)\cup Y\cup P_{R^{-1}}(Y)\subseteq Y$ and therefore it holds $(iv)$.

\noindent\emph{(iv) implies (i):}  Let us see (tc1). Let $x,y\in X$ such that $x\in R^{-1}(y)\cap Y$. Then $y\in P_{R^{-1}}(Y)$ and, by the hypothesis, $y\in Y$ and so there are $z=w=y\in R(x)\cap Y$ such that $z\leq y\leq w$. Similarly, we prove (tc2).
\end{proof}

\begin{lemma}\label{L43} Let $Y, Z$ be \tps-sets of $\mathfrak{X}$. Then, 
\begin{itemize}
    \item[(i)] $Y\cup Z$ is a \tps-set,
    \item[(ii)] if $Y$ is upward and $Z$ is downward, then $Y\cap Z$ is a  \tps-set,
    \item[(iii)] $\{x\}\in C_t(X)$ \ iff \ $R(x)=R^{-1}(x)=\emptyset$ or $R(x)=R^{-1}(x)=\{x\}$,
    \item[(iv)] if $\{x\}\in C_t(X)$ then $\{x\}^c$ is a \tps-set.
\end{itemize}
\end{lemma}

\

\subsection{Congruences of tense distributive lattices}
In what follows, $\mathcal{A}$ is a \tdlat-algebra. We denote by $Con(\mathcal{A})$ the lattice of congruences of the underlying lattice $\mathcal{A}_0$ and by $Con_t(\mathcal{A})$ the lattice of congruences of the \tdlat-algebra $\mathcal{A}$.  
\begin{lemma}\label{L44} $\Theta(Y)\in Con_t(\mathcal{A})$, for all $Y\in C_t(X(A))$. Where
\begin{equation}
  \Theta(Y)=\{(a,b)\in A\times A: \sigma_A(a)\cap Y=\sigma_A(b)\cap Y\}  
\end{equation}
\end{lemma}
\begin{proof} Let $Y\in C_{t}(X(A))$. Since $Y$ is closed, we have $\Theta(Y)\in Con(\mathcal{A})$. We only have to check that $\Theta(Y)$ is compatible with the tense operators. Let $(a,b)\in \Theta(Y)$, that is, $\sigma_A(a)\cap Y=\sigma_A(b)\cap Y$. We only prove it for the operators $\G$ and $\F$, the rest are similar using (tc2).\\[2mm] 
$\bullet$ $\sigma_A(\G a )\cap Y=\sigma_A(\G b )\cap Y$: \ Let $S\in \sigma_A(\G a)\cap Y$. Since $\sigma_A(\G a)=\G_{R_A}(\sigma_A(a))$, we have $R_A(S)\subseteq \sigma_A(a)$ and $S\in Y$. Then, $R_A(S)\subseteq \sigma_A(b)$, indeed, if $T\in R_A(S)$ then  $S\in R^{-1}_A(T)\cap Y$ and by  (tc1), there is $Z\in R_A(S)\cap Y$ such that  $Z\subseteq T$. Then, $Z\in R_A(S)\cap Y\subseteq \sigma_A(a)\cap Y=\sigma_A(b)\cap Y\subseteq \sigma_A(b)$ and since $\sigma_A(b)$ is upward, we have that $T\in \sigma_A(b)$. Therefore, $R_A(S)\subseteq \sigma_A(b)$ that is $S\in \G_{R_A}(\sigma_A(b))=\sigma_A(\G b)$ and then $S\in \sigma_A(\G b)\cap Y$. The other inclusion is similar.\\[2mm]
$\bullet$ $\sigma_A(\F a)\cap Y=\sigma_A(\F b)\cap Y$: \ Let $S\in \sigma_A(\F a)\cap Y$. Since $\sigma_A(\F a)=\F_{R_A}(\sigma_A(a))$ we have $R_A(S)\cap \sigma_A(a)\not=\emptyset$ and $S\in Y$. Then, there is $T\in R_A(S)$ such that $T\in \sigma_A(a)$. Then $S\in R^{-1}_A(T)\cap Y$ and by (tc1), there is $W\in R_A(S)\cap Y$ such that $T\subseteq W$ and taking into account that $\sigma_A(a)$ is upward, we have $W\in \sigma_A(a)\cap Y=\sigma_A(b)\cap Y\subseteq \sigma_A(b)$. Therefore,  $W\in R_A(S)\cap \sigma_A(b)$ and then $S\in \F_{R_A}(\sigma_A(b))=\sigma_A(\F b)$. That is, $S\in \sigma_A(\F b)\cap Y$. The other inclusion is similar.
\end{proof}

\

\begin{lemma}\label{L45} Let $\theta\in Con_t(\mathcal{A})$ and $q:A\longrightarrow A/\theta$ the natural epimorphism. Then,  $$Y=\{{\bf \Phi}(q)(S): S\in X(A/\theta)\}\in C_{t}(X(A)).$$
\end{lemma}
\begin{proof} From the fact that  $Con_t(\mathcal{A})$ is a sublattice of $Con(\mathcal{A})$ we have that $Y=\{{\bf \Phi}(q)(S): S\in X(A/\theta)\}$ is a closed set of $X(A)$ and $\theta=\Theta(Y)$. Besides, from Lemma \ref{L313}, ${\bf \Phi}(q):(X(A/\theta),R_{A/\theta})\longrightarrow(X(A),R_{A})$ is a \tps-function. Now, we check that $Y$ verifies (tc1) and (tc2). Let $T,Q\in X(A)$. \\[2mm]
(tc1) \ If $T\in R^{-1}_A(Q)\cap Y$ then there are $Z, W\in R_A(T)\cap Y$ such that  $Z\subseteq Q\subseteq W$. Let $T\in R^{-1}_A(Q)\cap Y$, then $T={\bf \Phi}(q)(S)$ for some $S\in X(A/\theta)$ since $T\in Y$ and therefore $Q\in R_A({\bf \Phi}(q)(S))$. By (tPSf2), there are $Z',W'\in R_{A/\theta}(S)$ such that ${\bf \Phi}(q)(Z')\subseteq Q\subseteq {\bf \Phi}(q)(W')$.
 From $Z',W'\in R_{A/\theta}(S)$ we know that ${\bf \Phi}(q)(Z'),{\bf \Phi}(q)(W')\in {\bf \Phi}(R_{A/\theta}(S))$ and by (tPSf1), we obtain ${\bf \Phi}(q)(Z'),{\bf \Phi}(q)(W')\in R_A({\bf \Phi}(q)(S))$ and therefore it is enough to take $Z={\bf \Phi}(q)(Z')$ and $W={\bf \Phi}(q)(W')$. \\[2mm]
(tc2): \ If $T\in R_A(Q)\cap Y$ then there are $Z, W\in R^{-1}_A(T)\cap Y$ such that $Z\subseteq Q\subseteq W$. The proof goes similarly to (tc1) using (tPSf1') and (tPSf3).
\end{proof}

\

\noindent From lemmas \ref{L44} and  \ref{L45} we obtain the next important result.

\begin{theorem}\label{T46} Let $\mathcal{A}$ be a \tdlat-algebra and ${\bf \Phi}(\mathcal{A})$ its associated \tps-space. Then, the lattice $C_{t}(X(A))$ of all closed \tps-sets of ${\bf \Phi}(\mathcal{A})$ is isomorphic to the dual lattice of  $Con_t(\mathcal{A})$ of all \tdlat-congruences of $\mathcal{A}$. Besides, the anti-isomorphism is the map $\varphi:C_{t}(X(A))\longrightarrow Con_t(\mathcal{A})$ defined by $\varphi(Y)=\Theta(Y)$.
\end{theorem}

\

\subsubsection{Congruences determined by tense filters (ideals)}



Now, we established the relation between \tdlat-filters and upward closed \tps-sets. For this, we need the following definition. Let $\mathcal{A}$ be a \tdlat-algebra and ${\bf \Phi}(\mathcal{A})$ its associated \tps-space.

\begin{definition} 
\begin{itemize}
\item[]
  \item[(a)] If $S$ is a \tps-filter of $\mathcal{A}$, define $\sigma(S)=\{T\in X(A): S\subseteq T\}$.
  \item[(b)] If $Y$ us an upward closed \tps-set of ${\bf \Phi}(\mathcal{A})$, define $\varrho(Y)=\{a\in A: Y\subseteq \sigma_A(a)\}$.
\end{itemize}
\end{definition}

\

\noindent Then

\begin{lemma}\label{L48} For each \tdlat-filter $S$ of $\mathcal{A}$ and for each upward closed \tps-set $Y$ of ${\bf \Phi}(\mathcal{A})$ the following conditions hold.
\begin{itemize}
  \item[(i)] $\sigma(S)$ is an upward closed \tps-set of ${\bf \Phi}(\mathcal{A})$ and $\varrho(\sigma(S))=S$.
  \item[(ii)] $\varrho(Y)$ is a \tdlat-filter of $\mathcal{A}$ and $\sigma(\varrho(Y))=Y$.
\end{itemize}
\end{lemma}
\begin{proof}
$(i)$: \ Let $S$ be a \tdlat-filter of $\mathcal{A}$. It is well-known that $\sigma(S)=\displaystyle\bigcap_{s\in S}\sigma_A(s)$ is an upward closed set of  ${\bf \Phi}(\mathcal{A})$ and $\varrho(\sigma(S))=S$. Besides, since $S$ is a \tdlat-filter, we have that $d(s)\in S$, for all $s\in S$, and then it is not difficult to check that $\sigma(S)=\displaystyle\bigcap_{s\in S}\sigma_A(d s)$. Therefore, by  Proposition \ref{P27} we have
\begin{eqnarray*}
\sigma(S)  & = & \bigcap_{s\in S}\sigma_A(\G s\wedge s\wedge \H s)\\ 
           & = & \bigcap_{s\in S}\left(\G_{R_A}(\sigma_A(s))\cap \sigma_A(s)\cap \H_{R^{-1}_A}(\sigma_A(s))\right)\\
           & = & \left(\bigcap_{s\in S}\G_{R_A}(\sigma_A(s))\right)\cap \left(\bigcap_{s\in S}\sigma_A(s)\right)\cap \left(\bigcap_{s\in S}\H_{R^{-1}_A}(\sigma_A(s))\right)\\
           & = & \G_{R_A}\left(\bigcap_{s\in S}\sigma_A(s)\right)\cap \left(\bigcap_{s\in S}\sigma_A(s)\right)\cap \H_{R^{-1}_A}\left(\bigcap_{s\in S}\sigma_A(s)\right)\\ 
           & = & \G_{R_A}(\sigma(S))\cap \sigma(S)\cap \H_{R^{-1}_A}(\sigma(S))
\end{eqnarray*}
From Lemma \ref{L42}, $\sigma(S)$ is a \tps-set.\\[2mm]
$(ii)$: \ Let $Y$ be an upward closed \tps-set in ${\bf \Phi}(\mathcal{A})$. We know that $\varrho(Y)=\displaystyle\bigcap_{T\in Y}T$ is a filter of $\mathcal{A}_0$ and  $\sigma(\varrho(Y))=Y$. So, we only have to check that the filter $\varrho(Y)$ is closed by $\G$ and $\H$. Let $x\in \varrho(Y)$ and suppose that $\G x\not\in \varrho(Y)$, then there is $T\in Y$ such that $\G x\not\in T$. By Lemma \ref{L225}, there is $T'\in X(A)$ such that $T'\in R_A(T)$ y $x\not\in T'$ and by Lemma \ref{L42} (tc3), $R_A(T)\subseteq Y$.Then $x\not\in \varrho(Y)$ which is a contradiction. Similarly, using (tc4), we prove that $\H x\in \varrho(Y)$, for all $x\in \varrho(Y)$.
\end{proof}

\

\noindent Denote by $\mathcal{F}_t(A)$ the lattice of all \tdlat-filters of $\mathcal{A}$ and by $C^{\uparrow}_{t}(X(A))$ the lattice of all upward closed \tps-set of ${\bf \Phi}(\mathcal{A})$. By Lemma \ref{L48}, it is immediate the following result.

\begin{theorem}\label{T49}  $C^{\uparrow}_{t}(X(A))$ is isomorphic to the dual lattice of $\mathcal{F}_t(A)$, and the isomorphism is given by the map  $Y\mapsto \varrho(Y)$ whose  inverse map is $S\mapsto \sigma(S)$.
\end{theorem}

\noindent It is well-known that every filter $S$ of a given bounded distributive lattice $\mathcal{L}$, the relation
\begin{equation}
\Theta_S=\{(a,b)\in A\times A: a\wedge s=b\wedge s\,\mbox{ for some }\,\,s\in S\}
\end{equation}

\noindent is a congruence of $\mathcal{L}$ and $\Theta_S=\Theta(\sigma(S))$. Besides, for each $Y$ is an upward closed set of $X(L)$ and $\Theta(Y)=\Theta_{\varrho(Y)}$.

\begin{remark} As it is shown in Lemma \ref{L48}, $\sigma(S)\in C_t(X(A))$ for all $S\in \mathcal{F}_t(A)$ and by Theorem \ref{T46}, we have that  $\Theta(\sigma(S))\in Con_t(A)$ and therefore $\Theta_S\in Con_t(A)$. \\[1.5mm]
It is worth mentioning that, from the algebraic point of view, there is no simple proof for this fact. However, using the equivalence proved in Section \ref{s2}, we have it almost immediately. This shows the importance of our development of this nice topological duality.
    
\end{remark}
 

\

\noindent Denote by $Con_{\mathcal{F}_t}(\mathcal{A})$ the lattice of all congruences of $\mathcal{A}$ determined by tense filters. From Theorems \ref{T46} and \ref{T49}, we can state.

\begin{theorem}\label{T411} $C^{\uparrow}_{t}(X(A))$ is isomorphic to the dual lattice of $Con_{\mathcal{F}_t}(\mathcal{A})$. The isomorphism is given by the map $Y\mapsto \Theta_{\varrho(Y)}$.
\end{theorem}

\

Analogously, we can establish a correspondence between \tdlat-ideals and downward closed \tps-sets.

\begin{definition} 
\begin{itemize}
\item[]
  \item[(a)] For each \tdlat-ideal $I$ of $\mathcal{A}$, define $\sigma(I)=\{T\in X(A): T\cap I=\emptyset\}$.
  \item[(b)] For each downward closed \tps-set $Z$ of ${\bf \Phi}(\mathcal{A})$, define $\varrho(Z)=\displaystyle\left(\bigcup_{T\in Z}T\right)^c$.
\end{itemize}
\end{definition}

\noindent Denote by $\mathcal{I}_t(A)$ the lattice of all \tdlat-ideals of $\mathcal{A}$  and by $C^{\downarrow}_{t}(X(A))$ the lattice of all downward closed \tps-sets of ${\bf \Phi}(\mathcal{A})$. Then

\begin{theorem}\label{T413}  $C^{\downarrow}_{t}(X(A))$ is isomorphic to the dual lattice of $\mathcal{I}_t(A)$ and the isomorphism is given by the map $Z\mapsto \varrho(Z)$ whose inverse map is $I\mapsto \sigma(I)$.
\end{theorem}

\noindent It is well-known that for every ideal $I$ of a bounded distributive lattice $\mathcal{L}$, the relation 
\begin{equation}
\Theta_I=\{(a,b)\in A\times A: a\vee i=b\vee i\,\mbox{ for some }\,\,i\in I\}
\end{equation}
is a congruence of $\mathcal{L}$. Besides, $\Theta_I=\Theta(\sigma(I))$ and, for every downward closed set $Z$ of $X(L)$, it holds \ $\Theta(Z)=\Theta_{\varrho(Z)}$.
Denote by  $Con_{\mathcal{I}_t}(\mathcal{A})$ the lattice of all \tdlat-congruences determined by \tdlat-ideals, by Theorems \ref{T46} and \ref{T413} we have

\begin{theorem}\label{T414}  $C^{\downarrow}_{t}(X(A))$ is isomorphic to the dual lattice of $Con_{\mathcal{I}_t}(\mathcal{A})$. The isomorphism is given by  $Z\mapsto \Theta_{\varrho(Z)}$.
\end{theorem}

\

\noindent We end this section summing up the most relevant relations obtained using the equivalence proved in  Section \ref{s3}.

\begin{eqnarray*}
    \tdlat-algebras & \longleftrightarrow & \tps-spaces\\
    \tdlat-homomorphisms  & \longleftrightarrow & \tps-functions\\
    \tdlat-filters & \longleftrightarrow & \tps-upward\,\,\,closed\,\,\,sets\\
    \tdlat-ideals & \longleftrightarrow & \tps-downward\,\,\,closed\,\,\,sets
\end{eqnarray*}  

\ 

\noindent Besides, we have established the following isomorphisms.

\begin{eqnarray*}
    Con_t(\mathcal{A}) & \longleftrightarrow & C_t(\mathfrak{X}(\mathcal{A}))\\
    Con_{\mathcal{F}_t}(\mathcal{A}) & \longleftrightarrow & C^{\uparrow}_t(\mathfrak{X}(\mathcal{A}))\\
    Con_{\mathcal{I}_t}(\mathcal{A}) & \longleftrightarrow & C^{\downarrow}_t(\mathfrak{X}(\mathcal{A}))    
\end{eqnarray*}  

\

\subsection{Simple and subdirectly irreducible objects in \tdlat}
Now, we apply the results obtained in the previous subsection in order to give a characterization of simple and subdirectly irreducible tense distributive lattices.

\begin{theorem}\label{T415} Let $\mathcal{A}$ be a \tdlat-algebra and let ${\bf \Phi}(\mathcal{A})$ be its associated \tps-space. The following conditions are equivalent.
\begin{itemize}
  \item[(i)] $\mathcal{A}$ is a simple \tdlat-algebra.
  \item[(ii)] $C_{t}(X(A))=\{\emptyset,X(A)\}$.
\end{itemize}
\end{theorem}
\begin{remark} In the \tdlat-algebra of Example \ref{ej2} one can verify that $C_{t}(X(A))=\{\emptyset,X(A)\}$ and, by Theorem \ref{T415}, it is a simple algebra.  
\end{remark}

\begin{theorem}\label{T417} Let $\mathcal{A}$ be a \tdlat-algebra and let ${\bf \Phi}(\mathcal{A})$ be its associated \tps-space. The following conditions are equivalent.
\begin{itemize}
  \item[(i)] $\mathcal{A}$ is a subdirectly irreducible  \tdlat-algebra.
  \item[(ii)] There is $Z\in C_{t}(X(A))\setminus\{X(A)\}$ such that $Y\subseteq Z$ for every $Y\in C_{t}(X(A))\setminus\{X(A)\}$.
\end{itemize}
\end{theorem}

\begin{lemma}\label{L418} For every \tdlat-algebra $\mathcal{A}$, the following conditions are equivalent.
\begin{itemize}
    \item[(a)] For every $X\in \mathcal{P}(A)\setminus\{\emptyset,\{0\},\{1\}\}$ and every  $a\in A$, there are $x_1,\hdots,x_{n_a},y_1,\hdots, y_{m_a}\in X$ and $p_a, q_a\in\omega$ such that $\displaystyle d^{p_a}\left(\bigwedge_{i=1}^{n_a}x_i\right)\leq a\leq\displaystyle \hat{d}^{p_a}\left(\bigvee_{j=1}^{m_a}y_j\right)$.
    \item[(b)] For every $a\in A\setminus\{0,1\}$ there are $p_a, q_a\in \omega$ such that $d^{p_a} a= 0$ and $\hat{d}^{q_a} a =1$. 
    \item[(c)] $\mathcal{F}_t(A)=\{A,\{1\}\}$ and \,$\mathcal{I}_t(A)=\{A,\{0\}\}$
    \item[] If $\mathcal{A}$ is a finite algebra, (a),(b) and (c) are equivalent to
    \item[(d)] $A^d=\{0,1\}$
\end{itemize}
\end{lemma}
\begin{proof}
\begin{itemize}
    \item[]
    \item[] (a) implies (b): Let $x\in A\setminus\{0,1\}$. It is enough to consider  $X=\{x\}$, and choosing $a=0$ we have that there is $p_a\in \omega$ such that $d^{p_a} a = 0$. Besides, choosing $a=1$ there is $q_a\in \omega$ such that $\hat{d}^{q_a} a =1$.
    \item[] (b) implies (c): Let $S\in \mathcal{F}_t(A)$ such that $S\neq\{1\}$. Then, there is $a\in S$ such that $a\neq1$ and, by hypothesis, there is $p_a\in \omega$ such that $d^{p_a} a=0$. Since $S$ is closed by $d$  we have $S=A$. Similarly, taking $I\in \mathcal{I}_t(A)$ such that $I\neq\{0\}$ we have $I=A$. 
    \item[] (c) implies (a): Let $X\in \mathcal{P}(A)\setminus\{\emptyset,\{0\},\{1\}\}$ and $a\in A$. Since $X\neq\{1\}$ and taking $\langle X\rangle$, the \tdlat-filter generated by $X$, we have that $\langle X\rangle=A$ and then  $a\in \langle X\rangle$. By Corollary \ref{C230}, there are $x_1,\hdots,x_{n_a}\in X$ and $p_a\in\omega$ such that $\displaystyle d^{p_a}\left(\bigwedge_{i=1}^{n_a}x_i\right)\leq a$. Analogously, since $X\neq\{0\}$ and considering $\langle X\rangle^{\partial}$, the \tdlat-ideal generated by $X$, we have $\langle X\rangle^{\partial}=A$. Then, by Lemma \ref{L236}, there are $y_1,\hdots, y_{m_a}\in X$ and $q_a\in\omega$ such that $ a\leq\displaystyle \hat{d}^{p_a}\left(\bigvee_{j=1}^{m_a}y_j\right)$. 
    \item[] Finally, if $\mathcal{A}$ es finite, from Corollary \ref{C233} we have that (d) is equivalent to (c) and, therefore, equivalent to (a) and (b).
\end{itemize}
\end{proof}

\begin{lemma}\label{L419} Let $\mathcal{A}$ be a simple \tdlat-algebra, then it satisfies conditions (a)-(d) from Lemma \ref{L418}.
\end{lemma}
\begin{proof} From Theorem \ref{T415} we know that $C_{t}(X(A))=\{\emptyset,X(A)\}$. Let us see that (a) from Lemma \ref{L418} holds. \\[2mm]
Let $X\in \mathcal{P}(A)\setminus\{\emptyset,\{0\},\{1\}\}$ and $a\in A$. On the one side, since $X\neq\emptyset$ and $X\neq\{1\}$, the \tdlat-filter generated by $X$ verifies $\{1\}\subset \langle X\rangle$ and, by Theorem \ref{T49}, we have $\sigma(\langle X\rangle)\subset \sigma(\{1\})=X(A)$ and $\sigma(\langle X\rangle)\in C_t(X(A))$. Therefore, $\sigma(\langle X\rangle)=\emptyset=\sigma(A)$ and then  $\langle X\rangle=A$. On the other side, since $X\neq\emptyset$ and $X\neq\{0\}$, similarly to the above, using Theorem \ref{T413} we have that the \tdlat-ideal generated by $X$ verifies $\langle X\rangle^{\partial}=A$. Then, $a\in \langle X\rangle$ and $a\in \langle X\rangle^{\partial}$.\\[2mm]
Finally, notice that condition (a) imply (d). 
\end{proof}

\noindent From Lemma \ref{L419}, we can state the next simple test to determine when a \tdlat-algebra is not simple. 
\begin{corollary} If $A^d\neq\{0,1\}$ then $\mathcal{A}$ is not simple.
\end{corollary} 


\subsection{Particular cases}\label{subsectPC}

We are going to describe the \tps-space associated to a \tdlat-algebra when its underlying lattice is a:
\begin{itemize}
    \item[(I)] Boolean algebra,
    \item[(II)] Heyting algebra, and
    \item[(III)] De Morgan algebra.
\end{itemize}

\begin{remark} It is not difficult to check that the category of tense algebras along with their homomorphisms, is a full subcategory of \tdlat. 
\end{remark}

\noindent{\bf (I)} Let $\mathcal{B}=\langle \mathcal{B}_0,\G,\H,\F,\P\rangle$ be a \tdlat-algebra where $\mathcal{B}_0$ is a Boolean algebra and ${\bf \Phi(\mathcal{B})}=(X(B),R_B)$ its associated \tps-space.  
\begin{itemize}
    \item In the Priestley space $X(B)$ the order $\subseteq$ is the trivial one. Therefore, $U\subseteq X(B)$ is upward (and downward). Besides, the topology  $\tau_{Pr}$ coincides with the Stone topology, that is, the topology with the subbase  $\Sigma=\{\sigma(a)\}_{a\in B}$. Therefore, $X(B)$ is a Boolean space.
    \item In this case, for all $S,T\in X(B)$ it is verified: 
    $$\G^{-1}(S)\subseteq T\,\,\, \mbox{ iff } \,\,\,T\subseteq \F^{-1}(S)\,\,\,  \mbox{ iff } \,\,\,\H^{-1}(T)\subseteq S\,\,\, \mbox{ iff } \,\,\,S\subseteq \P^{-1}(T)$$
    therefore, the relation $R_B$ is such that : $(S,T)\in R_B\,\,\, \mbox{ iff } \,\,\,\G^{-1}(S)\subseteq T$. Here, we re-obtain results in the literature (for instance, in \cite{GO, K}).
    \item The family of closed \tps-sets verifies: $C_t(X(B))=C^{\uparrow}_t(X(B))=C^{\downarrow}_t(X(B))$. Then, from Theorems  \ref{T46} and \ref{T49} we have the following well-known result.

\begin{theorem}\label{T422} The lattice $Con_t(\mathcal{B})$ of \tps-congruences of  $\mathcal{B}$ is isomorphic to the dual of the lattice of \tdlat-filters of $\mathcal{B}$.
\end{theorem}
Besides, from Theorems \ref{T46} and \ref{T413} 
\begin{theorem}\label{T423} The lattice $Con_t(\mathcal{B})$ of \tps-congruences of $\mathcal{B}$ is isomorphic to the dual of the lattice of \tdlat-ideals of $\mathcal{B}$.     
\end{theorem}
   \item Also, in this case, we re-obtain the known characterizations for simple and subdirectly irreducible algebras (see \cite{K}).
\begin{theorem}\label{T424} The following conditions are equivalent.
\begin{itemize}
    \item[(i)] $\mathcal{B}$ is simple.
    \item[(ii)] For all $a\in B\setminus\{1\}$ there is $p_a\in\omega$ such that $d^{p_a} a =0$.
    \item[(iii)] $\mathcal{F}_t(B)=\{B,\{1\}\}$ and \,$\mathcal{I}_t(B)=\{B,\{0\}\}$
    \item[] If $\mathcal{B}$ is finite (i),(ii) and (iii) are equivalent to 
    \item[(iv)] $B^d=\{0,1\}$.
\end{itemize}
\end{theorem}
\begin{theorem}\label{T425} The following conditions are equivalent.
\begin{itemize}
    \item[(i)] $\mathcal{B}$ is subdirectly irreducible.
    \item[(ii)] There is $b\in B\setminus\{1\}$ such that, for all $a\in B\setminus\{1\}$ there is $k_a\in\omega$ such that $d^{k_a} a \leq b$.
\end{itemize}
\end{theorem}
\end{itemize}

\noindent{\bf (II)} Let $\mathcal{A}=\langle \mathcal{A}_0,\G,\H,\F,\P\rangle$ be a \tdlat-algebra where $\mathcal{A}_0$ is a Heyting algebra and ${\bf \Phi(\mathcal{A})}=(X(A),R_A)$ its associated \tps-space.

\begin{itemize}
    \item In the Priestley space $X(A)$ it is verified that ${\downarrow}U$ is clopen, for all clopen subset $U$ of $X(A)$. Therefore, $X(A)$ is a Heyting space.
    \item Taking into account that congruence lattice of $\mathcal{A}_0$ is isomorphic to the lattice of all its filters, we can obtain results analogous to Theorems \ref{T422}, \ref{T424}, and \ref{T425}. Since $\mathcal{A}^d$ is a subalgebra of the Heyting algebra $\mathcal{A}_0$, we can give the next characterization for the finite case:
    \begin{theorem}\label{T426} If $\mathcal{A}$ is finite. The following conditions are equivalent.
\begin{itemize}
    \item[(i)] $\mathcal{A}$ is subdirectly irreducible.
    \item[(ii)] There is $u\in A^d\setminus\{1\}$ such that $a \leq u$ for all $a\in A^d\setminus\{1\}$.
    \item[(iii)] $\mathcal{A}^d$ is a subdirectly irreducible Heyting algebra.
\end{itemize}
\end{theorem}
\end{itemize}

\noindent{\bf (III)} Let $\mathcal{A}=\langle \mathcal{A}_0,\G,\H,\F,\P\rangle$ be a \tdlat-algebra where $\mathcal{A}_0$ is a De Morgan algebra where, $\F:={\sim}\G{\sim}$, $\P:={\sim}\H{\sim}$, and let ${\bf \Phi(\mathcal{A})}=(X(A),R_A)$ its associated \tps-space.

\begin{itemize}
    \item Taking into account that $\mathcal{A}_0$ is a De Morgan algebra we have that we can define on $X(A)$ the function $g_A:X(A)\to X(A)$ by $g_A(S)=\{x\in A: \,\sim x\not\in S\}$ and, therefore, $(X(A),g_A)$ is a De Morgan space. It is not difficult to check that: \begin{center}
    If $(S,T)\in R_A$ then $(g_A(S),g_A(T))\in R_A$.
    \end{center}
    From this we have that $(X(A),g_A,R_A)$ is a tense De Morgan space (see Definition 5.1 in \cite{F6})
    \item Also, we obtain the corresponding versions of Theorems  \ref{T46}, \ref{T415} and \ref{T417}.    
\end{itemize}

\section{Discrete duality}\label{s5}

Given a set $X$, we denote by $\mathcal{P}_i(X)$ the collection of upward subsets of the poset  $(\mathcal{P}(X), \subseteq)$.
In \cite{F6}, it was proved the following result.

\begin{lemma}(see~\cite{F6}) \label{L51} Let $(X,\leq)$ be a poset, $R$ a binary relation on $X$ and $R^{-1}$ the inverse relation of $R$. The following hold.
\begin{itemize}
  \item[\rm{(i)}] $R({\uparrow}x)\subseteq {\uparrow}R(x)$, for all $x\in X$ \, iff \, $\G_R(U)\in \mathcal{P}_i(X)$, for all $U\in \mathcal{P}_i(X)$.
  \item[\rm{(ii)}] $R^{-1}({\uparrow}x)\subseteq {\uparrow}R^{-1}(x)$, for all $x\in X$ \, iff \,  $\H_{R^{-1}}(U)\in \mathcal{P}_i(X)$, for all $U\in \mathcal{P}_i(X)$.
  \item[\rm{(iii)}] $R({\downarrow}x)\subseteq {\downarrow}R(x)$, for all $x\in X$ \, iff \,  $\F_R(U)\in \mathcal{P}_i(X)$, for all $U\in \mathcal{P}_i(X)$. 
  \item[\rm{(iv)}] $R^{-1}({\downarrow}x)\subseteq {\downarrow}R^{-1}(x)$, for all $x\in X$ \, iff \, $\P_{R^{-1}}(U)\in \mathcal{P}_i(X)$, for all $U\in \mathcal{P}_i(X)$. 
  \item[] where $\G_R, \H_{R^{-1}}, \F_R, \P_{R^{-1}}$ are as in Definition \ref{DGR}.
\end{itemize}
\end{lemma}
\noindent We introduce the notion of {\em tense distributive lattice} frame.

\begin{definition}[\tdlat-Frame]\label{D52} A structure ${\bf X}=(X,\leq, R)$ is an \tdlat-frame if $(X,\leq)$ is a poset and $R$ is a binary relation on $X$ such that for every $x\in X$ the following holds.
\begin{itemize}
  \item[\rm{(K1)}] $R({\uparrow}x)\subseteq {\uparrow}R(x)$,
  \item[\rm{(K2)}] $R^{-1}({\uparrow}x)\subseteq {\uparrow}R^{-1}(x)$,
  \item[\rm{(K3)}] $R({\downarrow}x)\subseteq {\downarrow}R(x)$, 
  \item[\rm{(K4)}] $R^{-1}({\downarrow}x)\subseteq {\downarrow}R^{-1}(x)$,
  \item[\rm{(K5)}] $R(x)={\uparrow}R(x)\cap{\downarrow}R(x)$.
\end{itemize}
\end{definition}

\begin{lemma} \label{L53} Let $(X,\leq)$ be a poset, let $R$ be a binary relation on $X$ and $R^{-1}$ its inverse. Then
\begin{itemize}
  \item[] the following conditions are equivalent:
  \item[\rm{(K3)}] $R({\downarrow}x)\subseteq {\downarrow}R(x)$, for all $x\in X$.
  \item[\rm{(K3)'}] ${\uparrow}R^{-1}(x)\subseteq R^{-1}({\uparrow}x)$, for all $x\in X$; and
  \item[] the following conditions are equivalent:
  \item[\rm{(K4)}] $R^{-1}({\downarrow}x)\subseteq {\downarrow}R^{-1}(x)$, for all $x\in X$.
  \item[\rm{(K4)'}] ${\uparrow}R(x)\subseteq R({\uparrow}x)$, for all $x\in X$.
\end{itemize}  
\end{lemma}

Taking into account Lemma \ref{L53}, we can restate the definition of \tdlat-frames as follows.

\begin{definition} \label{D54} A structure ${\bf X}=(X,\leq, R)$ is a \tdlat-frame if $(X,\leq)$ is a poset and $R$ is a binary relation on $X$ such that:
\begin{itemize}
  \item[(K1$^*$)] $R({\uparrow}x)={\uparrow}R(x)$, for all $x\in X$.
  \item[(K2$^*$)] $R^{-1}({\uparrow}x)={\uparrow}R^{-1}(x)$, for all $x\in X$
  \item[(K5)] $R(x)={\uparrow}R(x)\cap{\downarrow}R(x)$, for all $x\in X$ 
\end{itemize}
\end{definition}

\noindent Now we can introduce the notion of {\em complex algebra} associated with a given \tdlat-frame in this context.

\begin{definition}[Complex algebra]\label{D55} Let ${\bf X}=(X,\leq, R)$ be an \tdlat-frame. The complex algebra of ${\bf X}$ is the structure 
$$\mathfrak{C}({\bf X})=\langle \mathcal{P}_i(X),\cap, \cup, \G_R, \H_{R^{-1}}, \F_R, \P_{R^{-1}},\emptyset, X\rangle$$
where $\G_R, \H_{R^{-1}}, \F_R, \P_{R^{-1}}$ are as in Definition \ref{DGR}.
\end{definition}
\noindent Then, we have the following results.

\begin{lemma}\label{L56} The complex algebra of a given  \tdlat-frame is a complete tense distributive lattice.
\end{lemma}
\begin{proof} From Lemma \ref{L51} we know that the operators $\G_R, \H_{R^{-1}}, \F_R, \P_{R^{-1}}$ are well-defined. It is routine to check that axioms (t1)-(t8) are fulfilled. 
\end{proof}

\begin{proposition}\label{P57} If $\mathfrak{X}=(X,\leq,\tau,R)$ is a \tdlat-space, then ${\bf X}=(X,\leq,R)$ is an \tdlat-frame.
\end{proposition}
\begin{proof} (K5) is exactly (Lts2). Axioms (K1)-(K4) follow from (Lts3) and Lemma \ref{L51} taking into account that $D(X)\subseteq \mathcal{P}_i(X)$.
\end{proof}

\begin{remark} Let $\mathfrak{X}=(X,\leq,\tau,R)$ be an \tdlat-space and ${\bf X}=(X,\leq,R)$. Then, we can determine the following tense distributive lattices:
\begin{itemize}
  \item $\langle X, R\rangle^+=\langle \mathcal{P}(X),\cap, \cup, G_R, H_{R^{-1}}, F_R, P_{R^{-1}},\emptyset, X\rangle$
  \item $\mathfrak{C}({\bf X})=\langle \mathcal{P}_i(X),\cap, \cup, G_R, H_{R^{-1}}, F_R, P_{R^{-1}},\emptyset, X\rangle$
  \item ${\bf \Psi}(\mathfrak{X})=\langle D(X),\cap, \cup,G_R, H_{R^{-1}}, F_R, P_{R^{-1}},\emptyset,X\rangle$
\end{itemize}
Moreover, we have the following relation among them.
$${\bf \Psi}(\mathfrak{X}){\bf \leq} \mathfrak{C}({\bf X}){\bf \leq} \langle X, R\rangle^+$$
\end{remark}

\begin{definition}[Canonical frame]\label{D59}  The canonical frame of a given tense distributive lattice $\mathcal{A}$, is the structure $$\mathfrak{M}(\mathcal{A})=(X(A),\subseteq, R_A)$$
 where $R_A$ is the relation given in Definition \ref{DRA}.
\end{definition}
\noindent From Lemma \ref{L224}, it is immediate the following: 

\begin{lemma}\label{L510} The canonical frame of an \tdlat-algebra is an \tdlat-frame.
\end{lemma}

\noindent Besides, in a similar way to what was done in the proof of Lemma \ref{L312} we can prove the following.

\begin{lemma}\label{L511} Let $\mathcal{A}$ be a tense distributive lattice. Then $h_A:\mathcal{A}\longrightarrow \mathfrak{C}(\mathfrak{M}(\mathcal{A}))$ defined by  $h_A(a)=\{T\in X(A):\,a\in T\}$, for all $a\in A$, is an immersion of tense distributive lattices.
\end{lemma} 

\begin{definition} Let ${\bf X}=(X,\leq, R)$ and ${\bf X'}=(X',\leq', R')$ be two \tdlat-frames. An immersion from ${\bf X}$ into ${\bf X'}$ is a map $f:X\to X'$ such that for every $x,y\in X$ it holds:
\begin{itemize}
  \item[$\bullet$] $x\leq y$ \, iff \, $f(x)\leq' f(y)$. 
  \item[$\bullet$] $(x,y)\in R$ \, iff \, $(f(x),f(y))\in R'$. 
\end{itemize}
\end{definition}

\begin{lemma}\label{L513} Let {\bf X} be an \tdlat-frame, then $k_{X}:{\bf X}\longrightarrow \mathfrak{M}(\mathfrak{C}({\bf X}))$ defined by $k_{X}(x)=\{U\in \mathcal{P}_i(X): x\in U\}$ is an immersion from ${\bf X}$ into $\mathfrak{M}(\mathfrak{C}({\bf X}))$.
\end{lemma}
\begin{proof} It is a routine task to show that $k_X$ it is well define and, for every $x,y\in X$ we have
$$x\leq y\,\,\,\mbox{si y solo si}\,\,\,k_x(x)\subseteq k_x(y)$$
Notice that the following conditions are equivalent.
\begin{eqnarray*}
(k_{X}(x),k_{X}(y))\in R_{\mathfrak{C}(X)} & \mbox{ iff } & G^{-1}_R(k_{X}(x))\subseteq k_{X}(y)\,\,\&\,\,k_{X}(y)\subseteq F^{-1}_R(k_{X}(x))\\
     & \mbox{ iff } & (\forall U\in \mathcal{P}_i(X))(\,G_R(U)\in k_{X}(x) \mbox{ implies } U\in k_{X}(y))\\
     &     &  \hspace{4cm}\&\\
     &     & (\forall V\in \mathcal{P}_i(X))(\,V\in k_{X}(y) \mbox{ implies }F_R(U)\in k_{X}(x))\\
     & \mbox{ iff } & (\forall U\in \mathcal{P}_i(X))(\,x\in G_R(U)\,\,implica\,\,y\in U)\\
     &     &  \hspace{4cm}\&\\
     &     & (\forall V\in \mathcal{P}_i(X))(\,y\in V  \mbox{ implies } x\in F_R(V))\\
     & \mbox{ iff } & (1)(\forall U\in \mathcal{P}_i(X))(\,R(x)\subseteq U \mbox{ implies } y\in U)\\
     &     &  \hspace{4cm}\&\\
     &     & (2)(\forall V\in \mathcal{P}_i(X))(\,y\in V \mbox{ implies } R(x)\cap V\not=\emptyset)
\end{eqnarray*}
From the above, it is immediate that 
$$(x,y)\in R \mbox{ implies } (k_{X}(x),k_{X}(y))\in R_{\mathfrak{C}(X)}$$

\

For the converse, suppose that $(k_{X}(x),k_{X}(y))\in R_{\mathfrak{C}(X)}$. Then, conditions  (1) and (2) are clearly verified. Let $U=({\downarrow}y)^c$. It is clear that $U\in \mathcal{P}_i(X)$ and since $y\not\in U$ and (1), we have $R(x)\not\subseteq U$. That is, there is $z\in R(x)$ such that $z\not\in U$ and, from this, it holds $z\leq y$. Then, (3) holds and $y\in {\uparrow}R(x)$.
On the other hand, if $V={\uparrow}y$ then $y\in V$ and $V\in \mathcal{P}_i(X)$. Then, from (2) we have $R(x)\cap V \neq\emptyset$ and so there exists $w\in R(x)$ such that $y\leq w$. That is,  $y\in {\downarrow}R(x)$, from this and (3) we have that $$y\in {\uparrow}R(x)\cap{\downarrow}R(x)$$
since, from (K5), ${\uparrow}R(x)\cap{\downarrow}R(x)=R(x)$ we have $(x,y)\in R$.
\end{proof}

\

Lemmas \ref{L56}, \ref{L510}, \ref{L511}, \ref{L513} state a discrete duality between \tdlat-frames and tense distributive lattices that is summarized in the theorem.

\begin{theorem}[Discrete Duality] Let $\mathcal{A}$ be a tense distributive lattice and {\bf X} be an \tdlat-frame. Then

\begin{itemize}
  \item[(1)] The canonical map $\mathfrak{M}(\mathcal{A})$ of $\mathcal{A}$ is an \tdlat-frame.
  \item[(2)] The complex algebra $\mathfrak{C}({\bf X})$ of {\bf X} is a tense distributive lattice.
  \item[(3)] The map $h_A:\mathcal{A}\longrightarrow\mathfrak{C}(\mathfrak{M}(\mathcal{A}))$ defined by  $h_A(a)=\{T\in X(A):\,a\in T\}$ is an immersion of tense distributive lattices, and if $\mathcal{A}$ is finite, then $h_A$ is an isomorphism.
  \item[(4)] The map $k_{X}:{\bf X}\longrightarrow \mathfrak{M}(\mathfrak{C}({\bf X}))$ defined by $k_{X}(x)=\{U\in \mathcal{P}_i(X): x\in U\}$ is an immersion of \tdlat-frames and, if ${\bf X}$ is finite, then $k_X$ is an isomorphism.
\end{itemize}
\end{theorem}

\

It is possible to particularize these results to the cases of tense Heyting algebras and De Morgan algebras as it was showed in subsection \ref{subsectPC}.

\section{The logic that preserves degrees of truth w.r.t. \tdlat}\label{s6}

In this section we focus on the logic that preserves degrees of truth associated tense distributive lattices. Let $\La$ be a propositional signature and let $\K$ be a class of $\La$-algebras in which every algebra has an underlying structure of lattice with greatest element $1$. Let $\mathfrak{Fm}=\langle Fm, \La\rangle$ be the absolutely free $\La$-algebra of terms by a countably infinite set of propositional variables $Var$.

\begin{definition}\label{D61} The logic that preserve degrees of truth w.r.t. $\K$, is the propositional logic $\mathbb{L}_{{\K}}^{\leq}=\langle Fm, \models_{{\K}}^{\leq}\rangle$ defined as follows, for every $\Gamma\cup\{\alpha\}\subseteq Fm$:
\begin{itemize}
  \item[(i)] If $\Gamma$ is a finite non-empty set,
\begin{eqnarray*}
\Gamma \models_{\K}^{\leq} \alpha & \Longleftrightarrow & \hspace{1cm}\forall A\in{\K},\,\,\forall h\in Hom_{\K}(\mathfrak{Fm},A),\,\,\forall a\in A\\
& & \mbox{ if }h(\gamma)\geq a \mbox{ for every } \gamma\in\Gamma \mbox{ then } h(\alpha)\geq a.
\end{eqnarray*}
  \item[(ii)] $\emptyset\models_{{\bf K}}^{\leq} \alpha\,\,\,\Longleftrightarrow\,\,\,\forall A\in{\bf K},\,\,\forall h\in Hom(\mathfrak{Fm},A),\,\,h(\alpha)=1.$
  \item[(iii)] If $\Gamma$ is infinite,
\begin{eqnarray*}
\Gamma \models_{{\K}}^{\leq} \alpha & \Longleftrightarrow & \mbox{ there exists } \Gamma_0\subseteq \Gamma \mbox{ finite such that } \Gamma_0 \models_{{\K}}^{\leq} \alpha.\\
\end{eqnarray*}
 
\end{itemize}
\end{definition}

\begin{proposition}\label{P62} Let $\Gamma\cup\{\alpha\}\subseteq Fm$ be finite and non-empty. The following conditions are equivalent:
\begin{itemize}
  \item[(a)] $\Gamma \models_{{\K}}^{\leq} \alpha$
  \item[(b)] $\forall A\in{\K},\,\,\forall h\in Hom_{\K}(\mathfrak{Fm},A),\,\,\bigwedge\{h(\gamma):\gamma\in\Gamma\}\leq h(\alpha).$ 
\end{itemize}
\end{proposition}
\begin{proof} From Definition \ref{D61} and the fact that every algebra in $\K$ has an underlying structure of lattice.
\end{proof}

\begin{corollary}\label{C63} For $\{\alpha_1,\hdots,\alpha_n,\alpha\}\subseteq Fm,\,n\geq1$, the following conditions are equivalent:
\begin{itemize}
  \item[(i)] $\alpha_1,\hdots,\alpha_n \models_{{\K}}^{\leq} \alpha$,
  \item[(ii)] $\alpha_1\wedge\hdots\wedge\alpha_n \models_{{\K}}^{\leq} \alpha$. 
\end{itemize}
\end{corollary}

\begin{proposition}\label{P64} The consequence relation $\models_{{\K}}^{\leq}$ verifies the following properties. For every $\Gamma\cup\Delta\cup\{\alpha,\beta\}\subseteq Fm$:
\begin{itemize}
  \item[(C1)] $\alpha\in \Gamma$ implies $\Gamma\models_{{\K}}^{\leq}\alpha$, \hfill(Reflexivity)
  \item[(C2)] $\Delta\models_{{\K}}^{\leq}\alpha$, $\Delta\subseteq\Gamma$ implies $\Gamma\models_{{\K}}^{\leq}\alpha$, \hfill(Monotonic)   
  \item[(C3)] $\Delta\models_{{\K}}^{\leq}\alpha$, $\forall\beta\in\Delta(\Gamma\models_{{\K}}^{\leq}\beta)$ implies $\Gamma\models_{{\K}}^{\leq}\alpha$.\hfill(Cut for sets)
  \item[(C4)] $\Delta\models_{{\K}}^{\leq}\alpha$, $\Gamma,\alpha\models_{{\K}}^{\leq}\beta$ implies $\Delta,\Gamma\models_{{\K}}^{\leq}\beta$. \hfill(Cut for formulas)
  \item[(C5)] $\Gamma\models_{{\K}}^{\leq}\alpha$ implies $\varepsilon(\Gamma)\models_{{\K}}^{\leq}\varepsilon(\alpha)$, for every endomorphism $\varepsilon$ of $\mathfrak{Fm}$. \hfill(Structurality)
\end{itemize}

\end{proposition}
\begin{proof} From Proposition \ref{P62}.
\end{proof}

\begin{corollary} $\mathbb{L}_{{\bf K}}^{\leq}$ is a sentential, that is, $\models_{{\bf K}}^{\leq}$ is a finitary consequence relation on $Fm$.
\end{corollary}

Taking into account Proposition \ref{P62}, we can extend $\mathbb{L}_{{\bf K}}^{\leq}$ as follows: Let $\Gamma$ and $\Delta$ be two finite sets of formulas, we say that $\Delta$ \emph{is consequence of }  $\Gamma$ in $\mathbb{L}_{{\bf K}}^{\leq}$, denoted by $\Gamma\models_{{\bf K}}^{\leq}\Delta$, if 
$$\forall A\in{\bf K},\,\,\forall h\in Hom(\mathfrak{Fm},A),\,\,\bigwedge\{h(\gamma):\gamma\in\Gamma\}\leq \bigvee\{h(\delta):\delta\in\Delta\}.$$
If $\Delta$ is a singleton then we recover the consequence relation of Definition \ref{D61}.\\
In what follows, we consider the logic that preserves degrees of truth w.r.t. the class of tense distributive lattices, $\mathbb{L}_{\tdlat}^{\leq}=\langle Fm, \models_{\tdlat}^{\leq}\rangle$.

\subsection{A Gentzen-style system for $\mathbb{L}_{\tdlat}^{\leq}$}

In what follows, we deal with multiple-conclusioned sequent systems, i.e. sequents of the form $\Gamma\Rightarrow\Delta$ where $\Gamma$ are $\Delta$ finite sets of formulas. Besides, we denote by $\#\Gamma$ the set $\{\#\gamma: \gamma\in \Gamma\}$ for $\#\in\{ \H, \G, \F, \P\}$. Let $\mathfrak{Lt}$ be the system given by the following axioms and rules:

\

{\bf Axioms} 
$$\alpha\Rightarrow \alpha\hspace{2cm}\bot\Rightarrow\hspace{2cm}\Rightarrow\top$$

{\bf Structural rules}

$$[we]_i\,\,\frac{\Gamma\Rightarrow\Delta}{\Gamma,\alpha\Rightarrow\Delta}\hspace{1cm}[we]_d\,\,\frac{\Gamma\Rightarrow\Delta}{\Gamma\Rightarrow\Delta,\alpha}\hspace{1cm}[cut]\,\,\frac{\Gamma\Rightarrow\Delta,\alpha\hspace{.5cm}\alpha,\Gamma\Rightarrow\Delta}{\Gamma\Rightarrow\Delta}$$

\

{\bf  Logical rules}

\

\begin{minipage}[b]{0.5\textwidth}
\begin{itemize}
    \item[][$\wedge\Rightarrow$] \hspace{.6cm} $\displaystyle\frac{\Gamma,\alpha,\beta\Rightarrow\Delta}{\Gamma,\alpha\wedge\beta\Rightarrow\Delta}$     
    \item[][$\vee\Rightarrow$] \hspace{.7cm} $\displaystyle\frac{\Gamma,\alpha\Rightarrow\Delta\hspace{.5cm}\Gamma,\beta\Rightarrow\Delta}{\Gamma,\alpha\vee\beta\Rightarrow\Delta}$
    \item[][$\G^{\ast}$] \hspace{.7cm}$\displaystyle\frac{\Gamma\Rightarrow \Delta,\alpha}{\G\Gamma\Rightarrow \F\Delta,\G\alpha}$
    \item[][${\ast}\F$] \hspace{.7cm}$\displaystyle\frac{\Gamma,\alpha\Rightarrow \Delta}{\G\Gamma,\F\alpha\Rightarrow \F\Delta}$ 
    \item[][$\P\G$] \hspace{.7cm}$\displaystyle\frac{\alpha\Rightarrow \Delta}{\P\G\alpha\Rightarrow \Delta}$ 
    \item[][$\G\P$] \hspace{.7cm}$\displaystyle\frac{\Gamma\Rightarrow \alpha}{\Gamma\Rightarrow \G\P\alpha}$ 
\end{itemize}
\end{minipage} \hfill \begin{minipage}[b]{0.5\textwidth}
\begin{itemize}
    \item[][$\Rightarrow\wedge$]\hspace{.75cm} $\displaystyle\frac{\Gamma\Rightarrow\Delta,\alpha\hspace{.5cm}\Gamma\Rightarrow\Delta,\beta}{\Gamma\Rightarrow\Delta,\alpha\wedge\beta}$
    \item[][$\Rightarrow\vee$] \hspace{.75cm}$\displaystyle\frac{\Gamma\Rightarrow\Delta,\alpha,\beta}{\Gamma\Rightarrow\Delta,\alpha\vee\beta}$ 
    \item[] [$\H^{\ast}$] \hspace{.7cm}$\displaystyle\frac{\Gamma\Rightarrow \Delta,\alpha}{\H\Gamma\Rightarrow \P\Delta,\H\alpha}$
    \item[] [${\ast}\P$] \hspace{.7cm}$\displaystyle\frac{\Gamma,\alpha\Rightarrow \Delta}{\H\Gamma,\P\alpha\Rightarrow \P\Delta}$
    \item[] [$\F\H$] \hspace{.7cm}$\displaystyle\frac{\alpha\Rightarrow \Delta}{\F\H\alpha\Rightarrow \Delta}$
    \item[] [$\H\F$] \hspace{.7cm}$\displaystyle\frac{\Gamma\Rightarrow \alpha}{\Gamma\Rightarrow \H\F\alpha}$  
\end{itemize}
\end{minipage}

\

We write $\mathfrak{Lt}\vdash \Gamma\Rightarrow\Delta$ to indicate that there exists a $\mathfrak{Lt}$-proof of the sequent $\Gamma\Rightarrow\Delta$ and we say that it is provable in $\mathfrak{Lt}$. Besides, we denote $\Gamma\Leftrightarrow\Delta$ to denote that $\Gamma\Rightarrow\Delta$ and $\Delta\Rightarrow\Gamma$ are both provable in $\mathfrak{LT}$.

\begin{proposition}\label{P66} The following rules are derivable in $\mathfrak{Lt}$.
\end{proposition}
\begin{itemize}
  \item[] $[m\G]\,\,\,\,\displaystyle\frac{\Gamma\Rightarrow\alpha}{\G\Gamma\Rightarrow \G\alpha}$\hspace{3.5cm}$[m\H]\,\,\,\,\displaystyle\frac{\Gamma\Rightarrow\alpha}{\H\Gamma\Rightarrow \H\alpha}$ 
  \item[] $[m\F]\,\,\,\,\displaystyle\frac{\alpha\Rightarrow\Delta}{\F\alpha\Rightarrow \F\Delta}$\hspace{3.5cm}$[m\P]\,\,\,\,\displaystyle\frac{\alpha\Rightarrow\Delta}{\P\alpha\Rightarrow \P\Delta}$ 
  \item[] $[\G]\,\,\,\,\,\,\,\,\,\,\,\,\displaystyle\frac{\Rightarrow \alpha}{\Rightarrow \G\alpha}$\hspace{4.1cm}$[\H]\,\,\,\,\,\,\,\,\,\,\,\,\displaystyle\frac{\Rightarrow \alpha}{\Rightarrow \H\alpha}$
  \item[] $[\F]\,\,\,\,\,\,\,\,\,\,\,\,\displaystyle\frac{\alpha\Rightarrow }{\F\alpha\Rightarrow }$\hspace{4.1cm}$[\P]\,\,\,\,\,\,\,\,\,\,\,\,\displaystyle\frac{\alpha\Rightarrow }{\P\alpha\Rightarrow }$
  \item[] $[Ad\G]\,\,\,\displaystyle\frac{\P\alpha\Rightarrow \beta}{\alpha\Rightarrow \G\beta}$\hspace{3.8cm}$[Ad\H]\,\,\,\displaystyle\frac{\F\alpha\Rightarrow \beta}{\alpha\Rightarrow \H\beta}$
  \item[] $[Ad\P]\,\,\,\displaystyle\frac{\alpha\Rightarrow \G\beta}{\P\alpha\Rightarrow \beta}$\hspace{3.8cm}$[Ad\F]\,\,\,\displaystyle\frac{\alpha\Rightarrow \H\beta}{\F\alpha\Rightarrow \beta}$
\end{itemize}
\begin{proof} We show it just for $[Ad\P]$ and $[Ad\G]$

\

\begin{minipage}[b]{0.5\textwidth}
$[Ad\P]$
\begin{prooftree}
\AxiomC{$\alpha\Rightarrow \G\beta$}
\LeftLabel{\small{$[m\P]$}}
\UnaryInfC{$\P\alpha\Rightarrow \P\G\beta$}
\AxiomC{$\beta\Rightarrow \beta$}
\LeftLabel{\small{$[\P\G]$}}
\UnaryInfC{$\P\G\beta\Rightarrow \beta$}
\LeftLabel{\small{$[cut]$}}
\BinaryInfC{$\P\alpha\Rightarrow \beta$}
\end{prooftree}
\end{minipage} \hfill \begin{minipage}[b]{0.5\textwidth}
$[Ad\G]$
\begin{prooftree}
\AxiomC{$\alpha\Rightarrow\alpha$}
\LeftLabel{\small{$[\G\P]$}}
\UnaryInfC{$\alpha\Rightarrow \G\P\alpha$}
\AxiomC{$\P\alpha\Rightarrow \beta$}
\LeftLabel{\small{$[m\G]$}}
\UnaryInfC{$\G\P\alpha\Rightarrow \G\beta$}
\LeftLabel{\small{$[cut]$}}
\BinaryInfC{$\alpha\Rightarrow \G\beta$}
\end{prooftree}
\end{minipage}
\end{proof}

\begin{proposition}\label{P67} The following sequents are provable in $\mathfrak{Lt}$.
\begin{itemize}
  \item[(1)] $\G\top\Leftrightarrow \top$ and $\H\top\Leftrightarrow \top$,
  \item[(2)] $\G(\alpha\wedge\beta)\Leftrightarrow \G\alpha\wedge \G\beta$ and $\H(\alpha\wedge\beta)\Leftrightarrow \H\alpha\wedge \H\beta$,
  \item[(3)] $\alpha\Rightarrow  \G\P\alpha$ and $\alpha\Rightarrow \H\F\alpha$,
  \item[(4)] $\G(\alpha\vee\beta)\Rightarrow \G\alpha\vee \F\beta$ and $\H(\alpha\vee\beta)\Rightarrow \H\alpha\vee \P\beta$,
   \item[(5)] $\F\bot\Leftrightarrow \bot$ and $\P\bot\Leftrightarrow \bot$,
  \item[(6)] $\F(\alpha\vee\beta)\Leftrightarrow \F\alpha\vee \F\beta$ and $\P(\alpha\vee\beta)\Leftrightarrow \P\alpha\vee \P\beta$, 
  \item[(7)] $\P\G\alpha\Rightarrow \alpha$ and $\F\H\alpha\Rightarrow \alpha$,
  \item[(8)] $\G\alpha\wedge \F\beta\Rightarrow \F(\alpha\wedge\beta)$ and $\H\alpha\wedge \P\beta\Rightarrow \P(\alpha\wedge\beta)$,
  \item[(9)] $\G\alpha\Leftrightarrow \G\P\G\alpha$ and $\H\alpha\Leftrightarrow \H\F\H\alpha$,
  \item[(10)] $\F\alpha\Leftrightarrow \F\H\F\alpha$ and $\P\alpha\Leftrightarrow \P\G\P\alpha$
\end{itemize}
\end{proposition}
\begin{proof} We only prove (2), the others are similar.

\begin{prooftree}
\AxiomC{$\alpha\Rightarrow\alpha$}
\LeftLabel{\small{$[we]_i$}}
\UnaryInfC{$\alpha,\beta\Rightarrow\alpha$}
\LeftLabel{\small{$[\wedge\Rightarrow]$}}
\UnaryInfC{$\alpha\wedge\beta\Rightarrow\alpha$}
\LeftLabel{\small{$[m\G]$}}
\UnaryInfC{$\G(\alpha\wedge\beta)\Rightarrow \G\alpha$}

\AxiomC{$\beta\Rightarrow\beta$}
\LeftLabel{\small{$[we]_i$}}
\UnaryInfC{$\alpha,\beta\Rightarrow\beta$}
\LeftLabel{\small{$[\wedge\Rightarrow]$}}
\UnaryInfC{$\alpha\wedge\beta\Rightarrow\beta$}
\LeftLabel{\small{$[m\G]$}}
\UnaryInfC{$\G(\alpha\wedge\beta)\Rightarrow \G\beta$}

\LeftLabel{\small{$[\Rightarrow\wedge]$}}
\BinaryInfC{$\G(\alpha\wedge\beta)\Rightarrow \G\alpha\wedge \G\beta$}
\end{prooftree}

\begin{prooftree}
\AxiomC{$\alpha\Rightarrow \alpha$}
\LeftLabel{\small{$[we]_i$}}
\UnaryInfC{$\alpha,\beta\Rightarrow \alpha$}
\AxiomC{$\beta\Rightarrow \beta$}
\LeftLabel{\small{$[we]_i$}}
\UnaryInfC{$\alpha,\beta\Rightarrow \beta$}
\LeftLabel{\small{$[\Rightarrow\wedge]$}}
\BinaryInfC{$\alpha,\beta\Rightarrow \alpha\wedge\beta$}
\LeftLabel{\small{[$m\G]$}}
\UnaryInfC{$\G\alpha, \G\beta\Rightarrow \G(\alpha\wedge\beta)$}
\LeftLabel{\small{[$\wedge\Rightarrow]$}}
\UnaryInfC{$\G\alpha\wedge \G\beta\Rightarrow \G(\alpha\wedge\beta)$}

\end{prooftree}

\end{proof}

\subsection{Soundness and completeness}

Let $\Gamma$ and $\Delta$ be two finite subsets of $Fm$. We say that the sequent $\Gamma\Rightarrow\Delta$ of $\mathfrak{Lt}$ is  {\em valid} if it holds:
$$\forall A\in{\tdlat},\,\,\forall h\in Hom(\mathfrak{Fm},A),\,\,\bigwedge\{h(\gamma):\gamma\in\Gamma\}\leq \bigvee\{h(\delta):\delta\in\Delta\}.$$
Besides, we say that the rule of sequents
$$\frac{\Gamma_1\Rightarrow\Delta_1,\hdots,\Gamma_k\Rightarrow\Delta_k}{\Gamma\Rightarrow\Delta}$$
{\em preserves validity} if it holds:
$$\Gamma_i\Rightarrow\Delta_i  \mbox{ is valid for } 1\leq i\leq k \, \mbox{ implies } \, \Gamma\Rightarrow\Delta \mbox{ is valid. }$$
Then, the following proposition is immediate.

\begin{proposition}\label{P68} Let $\Gamma\Rightarrow\Delta$ be a sequent of $\mathfrak{Lt}$. The following conditions are equivalent:
\begin{itemize}
\item[(i)] $\Gamma\Rightarrow\Delta$ is valid,
\item[(ii)] $\Gamma\models_{{\tdlat}}^{\leq}\Delta$, that is,  $\Delta$ is consequence of $\Gamma$ in $\mathbb{L}_{{\tdlat}}^{\leq}$.
\end{itemize}

\end{proposition}

\

\begin{lemma}\label{L69} All the (structural and logical) rules of $\mathfrak{Lt}$ preserve validity.
\end{lemma}
\begin{proof} That structural rules preserve validity is an immediate consequence of the fact that every $A\in {\tdlat}$ is, in particular, is a distributive lattice. 
Similarly, we can check that rules $[\wedge\Rightarrow], [\Rightarrow\wedge], [\vee\Rightarrow]$ and $[\Rightarrow\vee]$ preserve validity.  

On the other hand, from axioms (t3) and (t7) we have that rules $[\G\P], [\H\F], [\P\G]$ and $[\F\H]$ preserve validity. Next, we just check t for $[\G^{\ast}]$, the others are similar.  
$$[\G^{\ast}] \hspace{.7cm}\frac{\Gamma\Rightarrow \Delta,\alpha}{\G\Gamma\Rightarrow \F\Delta,\G\alpha}$$
Suppose that $\Gamma\Rightarrow\Delta,\alpha$ is valid and let $A\in{\tdlat}$ and $h\in Hom(\mathfrak{Fm},A)$. Then
$$\bigwedge\{h(\gamma):\gamma\in\Gamma\}\leq \bigvee\{h(\delta):\delta\in\Delta\}\vee h(\alpha)$$
Then, by the monotonicity of $\G$ and axioms (t2), (t4) and (t6) we have
$$\bigwedge\{h(\G\gamma):\gamma\in\Gamma\}\leq \bigvee\{h(\F\delta):\delta\in\Delta\}\vee h(\G\alpha)$$
and therefore, the sequent $\G\Gamma\Rightarrow \F\Delta,\G\alpha$ is valid.
\end{proof}

\begin{theorem}[Soundness]\label{T610} Every sequent $\mathfrak{Lt}$-provable is valid.
\end{theorem}
\begin{proof} From the fact that the axioms of $\mathfrak{Lt}$ are valid and Lemma \ref{L69}.
\end{proof}

Next, we show the Completeness Theorem vía the construction of the Lindenbaum-Tarski algebra. Consinder the binary relations $\prec$ and $\equiv$ defined on $Fm$ as follows:
$$\alpha\prec \beta\hspace{.5cm}\mbox{ if and only if }\hspace{.5cm}\mathfrak{Lt}\vdash \alpha\Rightarrow \beta$$
$$\alpha\equiv \beta\hspace{.5cm}\mbox{ if and only if }\hspace{.5cm}\alpha\prec \beta\,\,\,y\,\,\,\beta\prec \alpha$$
Taking into account that $\mathfrak{Lt}\vdash \alpha\Rightarrow \alpha$ the rule $[cut]$ we have that $\prec$ is a pre-order and therefore $\equiv$ is a equivalence relation. Besides, using the rules $[\wedge\Rightarrow], [\Rightarrow\wedge], [\vee\Rightarrow]$ and $[\Rightarrow\vee]$ it is not difficult to verify that $\equiv$ is a compatible with 
 $\wedge$ and $\vee$. Finally, the compatibility of $\equiv$ w.r.t $\G,\H,\F$ and $\P$ follow from the rules $[m\G], [m\H], [m\F]$ and $[m\P]$, respectively. Hence, $\equiv$ is a congruence relation on the algebra $\mathfrak{Fm}$ and, from Proposition \ref{P67}, we have 

\begin{lemma}\label{L611} $\mathfrak{Fm}_{\equiv}=\langle Fm/{\equiv}, \wedge, \vee, \to, \neg, \G, \H, {\bf 0}, {\bf 1}\rangle$ is a \tdlat-algebra, where ${\bf 0}=[\bot]_{\equiv}$ and ${\bf 1}=[\top]_{\equiv}$.
\end{lemma}

\begin{lemma}\label{L612} Let $\Gamma\Rightarrow \Delta$ be a $\mathfrak{Lt}$-sequent, then the following conditions are equivalent.
\begin{itemize}
    \item[(i)] $\mathfrak{Lt}\vdash\Gamma\Rightarrow \Delta$ 
    \item[(ii)] $\mathfrak{Lt}\vdash\bigwedge\Gamma\Rightarrow \bigvee\Delta$
 \end{itemize}
\end{lemma}
\begin{proof} The proof is by induction on the number of elements of $\Gamma\cup\Delta$.
    
\end{proof}
\begin{theorem}[Completeness] \label{T613} Every valid sequent of $\mathfrak{Lt}$ is $\mathfrak{Lt}$-provable.
\end{theorem}
\begin{proof} Let $\Gamma\Rightarrow \Delta$ be a sequent of $\mathfrak{Lt}$ and suppose that $\mathfrak{Lt}\not\vdash\Gamma\Rightarrow \Delta$, that is , $\Gamma\Rightarrow \Delta$ is not provable in $\mathfrak{Lt}$. Let us denote by $\gamma$ and $\delta$ the formulas $\bigwedge\Gamma$ and $\bigvee\Delta$, respectively. Then, from Lemma \ref{L612} we have  $\mathfrak{Lt}\not\vdash\gamma\Rightarrow \delta$, or equivalently $[\gamma]_{\equiv}\not\leq [\delta]_{\equiv}$.\\
On the other hand, since $\mathfrak{Fm}_{\equiv}$ is a tense distributive lattice (Lemma \ref{L611}), consider the map $v:Var\to \mathfrak{Fm}_{\equiv}$ defined as $v(p)=[p]_{\equiv}$. It is clear that $v$ can be homomorphically extended to $\overline{v}:\mathfrak{Fm}\to \mathfrak{Fm}_{\equiv}$. Then,  $\overline{v}(\gamma)\not\leq\overline{v}(\delta)$, that is,
$$\bigwedge\{\overline{v}(\gamma):\gamma\in\Gamma\}\not\leq \bigvee\{\overline{v}(\delta):\delta\in\Delta\}$$
Therefore, $\Gamma\Rightarrow\Delta$ is not valid.
\end{proof}

\

\begin{corollary} The following conditions are equivalent.
\begin{itemize}
  \item[(i)] $\mathfrak{Lt}\vdash\Gamma\Rightarrow\Delta$, that is,  $\Gamma\Rightarrow\Delta$ is provable in $\mathfrak{Lt}$,
  \item[(ii)] $\Gamma\Rightarrow\Delta$ is valid,
  \item[(iii)] $\Gamma\models_{{\tdlat}}^{\leq}\Delta$, that is,  $\Delta$ is consequence of $\Gamma$ in $\mathbb{L}_{{\tdlat}}^{\leq}$.
\end{itemize}
\end{corollary}



\subsection{A Gentzen-style system for $\mathbb{L}_{\tdlat_{\bf B}}^{\leq}$, $\mathbb{L}_{\tdlat_{\bf H}}^{\leq}$ and $\mathbb{L}_{\tdlat_{\bf M}}^{\leq}$.}

\subsubsection{Classical case: A Gentzen-style system for $\mathbb{L}_{\tdlat_{\bf B}}^{\leq}$}

Let us consider the language $\mathcal{L}=\{\wedge, \vee, \to, \neg, \H, \G, \F, \P\}$ and let $\mathfrak{Lt_C}$ obtained from $\mathfrak{Lt}$ by adding the next rules:

\

\begin{minipage}[b]{0.5\textwidth}
\begin{itemize}
    \item[][$\neg\,\Rightarrow$] \hspace{.7cm}$\displaystyle\frac{\Gamma\Rightarrow\Delta,\alpha}{\neg\alpha,\Gamma\Rightarrow\Delta}$  
    \item[][$\to\Rightarrow$]\hspace{.5cm} $\displaystyle\frac{\Gamma\Rightarrow\Delta,\alpha\hspace{.5cm}\beta,\Gamma\Rightarrow\Delta}{\alpha\to\beta,\Gamma\Rightarrow\Delta}$
\end{itemize}
\end{minipage} \hfill \begin{minipage}[b]{0.5\textwidth}
\begin{itemize}
     \item[][$\Rightarrow\,\neg$]\hspace{.75cm} $\displaystyle\frac{\alpha,\Gamma\Rightarrow\Delta}{\Gamma\Rightarrow\Delta,\neg\alpha}$ 
    \item[][$\Rightarrow\,\to$]\hspace{.5cm}$\displaystyle\frac{\alpha,\Gamma\Rightarrow\Delta,\beta}{\Gamma\Rightarrow\Delta,\alpha\to\beta}$ 
\end{itemize}
\end{minipage}

\begin{lemma} The following sequents are provable in $\mathfrak{Lt_C}$.
\begin{itemize}
  \item[(0)] Every $\mathfrak{Lt}$-provable sequent,
  \item[(1)] $\F\alpha\Leftrightarrow \neg\G\neg \alpha$\, and\, $\P\alpha\Leftrightarrow \neg\H\neg \alpha$,
  \item[(2)] $\G\alpha\Leftrightarrow \neg\F\neg \alpha$\, and\, $\H\alpha\Leftrightarrow \neg\P\neg \alpha$,
  \item[(3)] $\alpha\Rightarrow\G\neg\H\neg\alpha$\, and\, $\alpha\Rightarrow\H\neg\G\neg\alpha$
  \item[(4)] $\G(\alpha\to\beta)\Rightarrow\G\alpha\to\G\beta$\, and\, $\H(\alpha\to\beta)\Rightarrow\H\alpha\to\H\beta$,
\end{itemize}
\end{lemma}
\begin{proof} We just show  (4)

$(4)$

\begin{prooftree}
\AxiomC{$\alpha\Rightarrow\alpha$}
\LeftLabel{\small{$[we]_d$}}
\UnaryInfC{$\alpha\Rightarrow\beta,\alpha$}
\AxiomC{$\beta\Rightarrow\beta$}
\LeftLabel{\small{$[we]_i$}}
\UnaryInfC{$\beta,\alpha\Rightarrow\beta$}
\LeftLabel{\small{$[\to\Rightarrow]$}}
\BinaryInfC{$\alpha\to\beta,\alpha\Rightarrow\beta$}
\LeftLabel{\small{$[m\G]$}}
\UnaryInfC{$\G(\alpha\to\beta),\G\alpha\Rightarrow\G\beta$}
\LeftLabel{\small{$[\Rightarrow\to]$}}
\UnaryInfC{$\G(\alpha\to\beta)\Rightarrow\G\alpha\to\G\beta$}

\end{prooftree}

\end{proof}

\begin{theorem} $\mathfrak{Lt_C}$ is sound and complete w.r.t. the class $\tdlat_{\bf B}$.
\end{theorem}

\begin{theorem} $\mathfrak{Lt_C}$ is an alternative presentation for $K_t$.
\end{theorem}

\subsubsection{Intuitionistic case: A Gentzen-style system for $\mathbb{L}_{\tdlat_{\bf H}}^{\leq}$}

Consider the language $\mathcal{L}=\{\wedge, \vee, \to, \neg, \H, \G, \F, \P\}$ and let $\mathfrak{Lt_I}$ the system obtained from $\mathfrak{Lt}$ by adding the next rules:

\

\begin{minipage}[b]{0.5\textwidth}
\begin{itemize}
    \item[][$\neg\,\Rightarrow$] \hspace{.7cm}$\displaystyle\frac{\Gamma\Rightarrow\alpha}{\neg\alpha,\Gamma\Rightarrow}$  
    \item[][$\to\Rightarrow$]\hspace{.5cm} $\displaystyle\frac{\Gamma\Rightarrow\Delta,\alpha\hspace{.5cm}\beta,\Gamma\Rightarrow\Delta}{\alpha\to\beta,\Gamma\Rightarrow\Delta}$
\end{itemize}
\end{minipage} \hfill \begin{minipage}[b]{0.5\textwidth}
\begin{itemize}
     \item[][$\Rightarrow\,\neg$]\hspace{.75cm} $\displaystyle\frac{\alpha,\Gamma\Rightarrow}{\Gamma\Rightarrow\neg\alpha}$ 
    \item[][$\Rightarrow\,\to$]\hspace{.5cm}$\displaystyle\frac{\alpha,\Gamma\Rightarrow\Delta,\beta}{\Gamma\Rightarrow\Delta,\alpha\to\beta}$ 
\end{itemize}
\end{minipage}

\begin{lemma} The following sequents are $\mathfrak{Lt_I}$-provable.
\begin{itemize}
  \item[(0)] Every $\mathfrak{Lt}$-provable sequent,
  \item[(1)] $\F(\alpha\to\beta)\Rightarrow \G\alpha\to\F\beta$\,\,and\,\, $\P(\alpha\to\beta)\Rightarrow \H\alpha\to\P\beta$,
  \item[(2)] $\H\F\alpha\to\P\G\beta\Rightarrow\alpha\to\beta$\,\,and\,\, $\G\P\alpha\to\F\H\beta\Rightarrow\alpha\to\beta$,
  \item[(3)] $\F\alpha\to\G\beta\Rightarrow\G(\alpha\to\beta)$\,\,and\,\, $\P\alpha\to\H\beta\Rightarrow\H(\alpha\to\beta)$.
\end{itemize}
\end{lemma}
\begin{proof}

$(1)$
\begin{prooftree}
\AxiomC{$\alpha\Rightarrow\alpha$}

\AxiomC{$\beta\Rightarrow\beta$}
\LeftLabel{\small{$[we]_i$}}
\UnaryInfC{$\beta,\alpha\Rightarrow\beta$}
\LeftLabel{\small{$[\to\Rightarrow]$}}
\BinaryInfC{$\alpha,\alpha\to\beta\Rightarrow\beta$}
\LeftLabel{\small{$[{\ast}\F]$}}
\UnaryInfC{$\G\alpha,\F(\alpha\to\beta)\Rightarrow\F\beta$}
\LeftLabel{\small{$[\Rightarrow\to]$}}
\UnaryInfC{$\F(\alpha\to\beta)\Rightarrow\G\alpha\to\F\beta$}
\end{prooftree}

\

$(2)$

\begin{prooftree}
\AxiomC{$\alpha\Rightarrow\alpha$}
\LeftLabel{\small{$[\H\F]$}}
\UnaryInfC{$\alpha\Rightarrow\H\F\alpha$}

\AxiomC{$\P\G\beta\Rightarrow\P\G\beta$}
\LeftLabel{\small{$[we]_i$}}
\UnaryInfC{$\P\G\beta,\alpha\Rightarrow\P\G\beta$}
\LeftLabel{\small{$[\to\Rightarrow]$}}
\BinaryInfC{$\alpha, \H\F\alpha\to\P\G\beta\Rightarrow\P\G\beta$}
\LeftLabel{\small{$[\Rightarrow\to]$}}
\UnaryInfC{$\H\F\alpha\to\P\G\beta\Rightarrow\alpha\to\P\G\beta$}

\AxiomC{$\alpha\Rightarrow\alpha$}

\AxiomC{$\beta\Rightarrow\beta$}
\LeftLabel{\small{$[\P\G]$}}
\UnaryInfC{$\P\G\beta\Rightarrow\beta$}
\LeftLabel{\small{$[we]_i$}}
\UnaryInfC{$\P\G\beta,\alpha\Rightarrow\beta$}
\LeftLabel{\small{$[\to\Rightarrow]$}}
\BinaryInfC{$\alpha,\alpha\to\P\G\beta\Rightarrow\beta$}
\LeftLabel{\small{$[\Rightarrow\to]$}}
\UnaryInfC{$\alpha\to\P\G\beta\Rightarrow\alpha\to\beta$}
\LeftLabel{\small{$[cut]$}}
\BinaryInfC{$\H\F\alpha\to\P\G\beta\Rightarrow\alpha\to\beta$}
\end{prooftree}

\

$(3)$

\begin{prooftree}
\AxiomC{$\P(\F\alpha\to\G\beta)\Rightarrow\H\F\alpha\to\P\G\beta\,\,\,(i)$}

\AxiomC{$\H\F\alpha\to\P\G\beta\Rightarrow\alpha\to\beta\,\,\,(ii)$}
\LeftLabel{\small{[cut]}}
\BinaryInfC{$\P(\F\alpha\to\G\beta)\Rightarrow\alpha\to\beta$}
\LeftLabel{\small{$[Ad\G]$}}
\UnaryInfC{$\F\alpha\to\G\beta\Rightarrow\G(\alpha\to\beta)$}

\end{prooftree}

\end{proof}

\begin{theorem} $\mathfrak{Lt_I}$ is sound and complete w.r.t. the class $\tdlat_{\bf H}$.
\end{theorem}

\begin{theorem} $\mathfrak{Lt_I}$ is an alternative syntactic presentation for $IK_t^*$.
\end{theorem}

\subsubsection{The De Morgan case: A Gentzen-style system for $\mathbb{L}_{\tdlat_{\bf M}}^{\leq}$}

Consider $\mathcal{L}=\{\wedge, \vee, \sim, \H, \G\}$ and consider the following definitions:

$$\F:={\sim}\G{\sim}\hspace{2cm}and\hspace{2cm}\P:={\sim}\H{\sim}$$

Let $\mathfrak{Lt_{DM}}$ the system obtained from $\mathfrak{Lt}$ by adding the following rules:

$$[\sim]\hspace{.5cm} \displaystyle\frac{\alpha\Rightarrow\beta}{\sim\beta\Rightarrow\,\sim\alpha}$$

\

\begin{minipage}[b]{0.5\textwidth}
\begin{itemize}
    \item[][$\sim\sim\,\Rightarrow$] \hspace{.7cm}$\displaystyle\frac{\Gamma,\alpha\Rightarrow\Delta}{\Gamma,\sim\sim\alpha\Rightarrow\Delta}$  \end{itemize}
\end{minipage} \hfill \begin{minipage}[b]{0.5\textwidth}
\begin{itemize}
     \item[][$\Rightarrow\,\sim\sim$]\hspace{.75cm} $\displaystyle\frac{\Gamma\Rightarrow\Delta,\alpha}{\Gamma\Rightarrow\Delta,\sim\sim\alpha}$ 
\end{itemize}
\end{minipage}

\begin{theorem} $\mathfrak{Lt_{DM}}$ is sound and complete w.r.t. to the class $\tdlat_{\bf M}$.
\end{theorem}

\begin{theorem} $\mathfrak{Lt_{DM}}$ is an alternative presentation for the system  $DMt$ of \cite{FL}.
\end{theorem}


\subsection{Relational semantics for $\mathfrak{Lt}$ by \tdlat-frames} 

We denote by $\mathfrak{M}_{Lt}$ the class formed by all the \tdlat-frames (Definition \ref{D52}). 

\begin{definition} A model based in a \tdlat-frame ${\bf X}=(X,\leq,R)$ is a system $\mathcal{M}=({\bf X},m)$ where $m:Var\to \mathcal{P}_i(X)$ is a map called {\em meaning map}. 
\end{definition}

\begin{definition} Let $\mathcal{M}=({\bf X},m)$ be a model based on the \tdlat-frame {\bf X} defined by the satisfaction relation $\models\subseteq X\times Fm$ recursively. We say that the model the state $x\in X$ in the model $\mathcal{M}$ satisfies the formula $\alpha\in Fm$, denoted by $\mathcal{M},x\models \alpha$, if the following conditions are fulfilled:
\begin{itemize}
  \item[(1)] $\mathcal{M},x\models \top$.
  \item[(2)] $\mathcal{M},x\not\models \bot$.
  \item[(3)] $\mathcal{M},x\models p$ \, iff \, $x\in m(p)$, for $p\in Var$.
  \item[(4)] $\mathcal{M},x\models \beta\wedge\gamma$ \, iff \, $\mathcal{M},x\models \beta\,\,\,\&\,\,\,\mathcal{M},x\models \gamma$.
  \item[(5)] $\mathcal{M},x\models \beta\vee\gamma$ \, iff \, $\mathcal{M},x\models \beta$ or $\mathcal{M},x\models \gamma$
  \item[(6)] $\mathcal{M},x\models \G\beta$ \, iff \, for every $x'\in X$, if $x'\in R(x)$ then $\mathcal{M},x'\models \beta$.
  \item[(7)] $\mathcal{M},x\models \H\beta$ \, iff \, for every $x'\in X$, if $x'\in R^{-1}(x)$ then $\mathcal{M},x'\models \beta$.
  \item[(8)] $\mathcal{M},x\models \F\beta$ \, iff \, there is $x'\in X$ such that $x'\in R(x)$ and $\mathcal{M},x'\models \beta$.
  \item[(9)] $\mathcal{M},x\models \P\beta$ \, iff \, there is $x'\in X$ such that $x'\in R^{-1}(x)$ and $\mathcal{M},x'\models \beta$.
\end{itemize}
\end{definition}

\

\begin{remark} Given a model $\mathcal{M}=({\bf X},m)$ based on the \tdlat-frame {\bf X}, since $m:Var\to \mathcal{P}_i(X)$ there exists a unique homomorphic extension $\overline{m}:\mathfrak{Fm}\to \mathfrak{C}({\bf X})$ where $\mathfrak{C}({\bf X})$ is the complex algebra of {\bf X} (Definition \ref{D55} ). Besides, for every $\alpha\in Fm$, $\overline{m}(\alpha)=\{x\in X: \mathcal{M},x\models \alpha\}$ (this can be proved by induction on the complexity of the formula $\alpha$).
\end{remark}

\

\begin{definition} We say that the sequent $\Gamma\Rightarrow\Delta$ of $\mathfrak{Lt}$ is
\begin{itemize}
  \item[$\bullet$] valid in the model $\mathcal{M}=({\bf X},m)$ based on the \tdlat-frame {\bf X} iff \, $\overline{m}(\bigwedge\Gamma)\subseteq \overline{m}(\bigvee\Delta)$,
  \item[$\bullet$] valid in the \tdlat-frame {\bf X} iff it is valid in every model $\mathcal{M}=({\bf X},m)$ based in {\bf X}.
  \item[$\bullet$] valid in $\mathfrak{M}_{Lt}$ iff it is valid in every ${\bf X}\in \mathfrak{M}_{Lt}$ and, in this case, we write  $\mathfrak{M}_{Lt}\models\Gamma\Rightarrow\Delta$.
\end{itemize}
\end{definition}

\

\begin{theorem} Let $\Gamma\Rightarrow\Delta$ be a sequent of $\mathfrak{Lt}$. The following conditions are equivalent. 
\begin{itemize}
  \item[(i)] $\mathfrak{Lt}\vdash\Gamma\Rightarrow\Delta$, that is, $\Gamma\Rightarrow\Delta$ is provable in $\mathfrak{Lt}$.
  \item[(ii)] $\mathfrak{M}_{Lt}\models\Gamma\Rightarrow\Delta$, that is, $\Gamma\Rightarrow\Delta$ is valid in every ${\bf X}\in \mathfrak{M}_{Lt}$.
\end{itemize}
\end{theorem}
\begin{proof}
$(i)$ implies $(ii)$: Suppose that $\mathfrak{M}_{Lt}\not\models\Gamma\Rightarrow\Delta$, that is, there is a model $\mathcal{M}=({\bf X},m)$ based on the \tdlat-frame {\bf X} such that $\overline{m}(\bigwedge\Gamma)\not\subseteq \overline{m}(\bigvee\Delta)$. Then, there exists the \tdlat-algebra $\mathfrak{C}({\bf X})$ (Lemma \ref{L56}) and $\overline{m}\in Hom(\mathfrak{Fm},\mathfrak{C}({\bf X}))$ such that $\bigwedge\{\overline{m}(\gamma): \gamma \in \Gamma\}\not\subseteq \bigvee\{\overline{m}(\delta): \delta \in \Delta\}$ and therefore $\Gamma\Rightarrow\Delta$ is not valid. Then, by Theorem \ref{T610}, we have that $\mathfrak{Lt}\not\vdash\Gamma\Rightarrow\Delta$. \\[2mm]
$(ii)$ implies $(i)$: Suppose that $\mathfrak{Lt}\not\vdash\Gamma\Rightarrow\Delta$ then, by Theorem \ref{T613}, $\Gamma\Rightarrow\Delta$ is not valid. Therefore, there exists a \tdlat-algebra $\mathcal{A}$ and homomorphism $f:\mathfrak{Fm}\to\mathcal{A}$ such that $f(\bigwedge\Gamma)\not\leq f(\bigvee\Delta)$. Then, there is $S\in X(A)$ such that $f(\bigwedge\Gamma)\in S$ and $f(\bigvee\Delta)\not\in S$. \\
On the other hand, from Lemmas \ref{L510} and \ref{L511} we know that $\mathfrak{M}(\mathcal{A})$ is a \tdlat-frame and $h_A:\mathcal{A}\to \mathfrak{C}(\mathfrak{M}(\mathcal{A}))$ defined by $h_A(a)=\{T\in X(A): a\in T\}$ is an immersion of \tdlat-algebras. Then, there is a model $\mathcal{M}=(\mathfrak{M}(\mathcal{A}),m)$ based in $\mathfrak{M}(\mathcal{A})$ where $m=h_A\circ f$ and such that $m(\bigwedge\Gamma)\not\subseteq m(\bigvee\Delta)$. Therefore,  $\mathfrak{M}_{Lt}\not\models\Gamma\Rightarrow\Delta$.
\end{proof}

\

\begin{corollary} The following conditions are equivalent.
\begin{itemize}
  \item[(i)]  $\Gamma\Rightarrow\Delta$ is provable in $\mathfrak{Lt}$.
  \item[(ii)] $\Delta$ is consequence of $\Gamma$ in $\mathbb{L}_{{\bf tdLAt}}^{\leq}$.
  \item[(iii)] $\Gamma\Rightarrow\Delta$ is valid in every \tdlat-frame {\bf X}.  
\end{itemize}
\end{corollary}

Again, it is possible to particularize these results and obtain relational semantics for the different logics presented in this section. We leave the details to the reader.


\end{document}